\documentclass{amsart}
	\usepackage[utf8]{inputenc}
	\usepackage[pdftex]{graphicx}
	\usepackage{epstopdf}
	\usepackage{graphics,psfrag,subfigure,epsfig}
	\usepackage{caption}
	\usepackage{amssymb}
        \usepackage{mathabx}
	\usepackage{latexsym}
	\usepackage{xypic}
	\usepackage{multirow}
	\usepackage{xcolor}
        \usepackage{tikz-cd}

	\newtheorem{theorem}{Theorem}[section]
	\newtheorem{lemma}[theorem]{Lemma}
	\newtheorem{proposition}[theorem]{Proposition}
	
    	\newtheorem{corollary}[theorem]{Corollary}

	\theoremstyle{definition}
	\newtheorem{definition}[theorem]{Definition}

	\theoremstyle{remark}

	\numberwithin{equation}{section}

	\def\deg{{\rm deg}}

	\def\C{{\mathbb C}}

	\def\P{\mathbb P_{\mathbb C}}

	\title[Exceptional Curves]{Exceptional algebraic curves  for infinite  subgroups of ${\rm PGL}(n+1,\mathbb{C})$} 
	
	\author{Angel Cano}
	\address{ UCIM UNAM, Av. Universidad s/n. Col. Lomas de Chamilpa, C.P. 62210, Cuernavaca, Morelos, México.}
	\email{angelcano@im.unam.mx}

	\author{ Rodrigo Davila }
	\address{ P O Box 500701, Office No. 108, Block 3, Dubai Knowledge Park, Dubai.}
	\email{r.figueroa@educationpathways.ac.ae }
	
	\author{ Luis Loeza }
	\address{ IIT UACJ, Av. del Charro no. 610 Nte. Col. Partido Romero CP 32310,  Ciudad Juárez, Chihuahua, México.}
	\email{luis.loeza@uacj.mx}


\begin{document}

		\begin{abstract}
We classify algebraic curves in $\mathbb{CP}^{n}$ ($n \geq 2$) that are invariant under an infinite subgroup of $\operatorname{PGL}(n+1,\mathbb{C})$. In particular, we prove that any irreducible,
non-degenerate, one-dimensional algebraic set in $\mathbb{CP}^{n}$ invariant under an infinite subgroup of $\operatorname{Aut}(\mathbb{CP}^{n})$  must be projectively equivalent to a monomial curve.			
		\end{abstract}

		\maketitle

\section{Introduction}

In the 1990s, foundational work by J. Seade and A. Verjovsky~\cite {SV} established deep connections between discrete group actions on $\mathbb{CP}^n$, foliation theory, and the iteration of holomorphic endomorphisms. This framework was later extended by an analog of Sullivan's dictionary in dimension two~\cite{BCNS}, bridging the dynamics of holomorphic endomorphisms of $\mathbb{CP}^2$ with discrete groups of biholomorphisms. Motivated by these ideas, we investigate invariant algebraic varieties under projective transformations, with a focus on curves. Our results align closely with the rigidity principles of the Matsumura-Monsky theorem~\cite{MM}, which constrains automorphism groups of projective varieties, revealing analogous constraints on invariant sub-varieties. 

Building on this interplay between group actions and geometric rigidity, our work also connects to a central conjecture in iteration dynamics, positing that invariant sub-varieties of endomorphisms $f \colon \mathbb{CP}^n \to \mathbb{CP}^n$ must be linear subspaces. Although this has been established for dimensions~1 and~2~\cite{BCS}, degree-$(n+1)$ divisors~\cite[Theorem~2.1]{HN}, and smooth hypersurfaces~\cite{CN}, the general case of the higher dimension remains open; for recent progress, see H\"oring~\cite{Ho}. In this article, we prove:

\begin{theorem} \label{t:main1}
Let \( g \in \mathrm{PGL}(n+1,\mathbb{C}) \), \( n \geq 2 \), be an element of infinite order, and let \( S \subset \mathbb{CP}^n \) be a one-dimensional algebraic set invariant under \( g \). For every irreducible component \( S_0 \subset S \), one of the following holds:
\begin{enumerate}
    \item  After a coordinate change, \( S_0 \) is parametrized by a monomial curve, {\it i.e.} it coincides with the image of
    \[
        \xi_{\mathbf{k}} \colon \mathbb{CP}^1 \to \mathbb{CP}^n, \quad [z:w] \mapsto \left[z^{k_1}: z^{k_2}w^{k_1 - k_2}: \cdots : z^{k_m}w^{k_1 - k_m}: w^{k_1}: \mathbf{0}\right],
    \]
    where \( 1 < m \leq n \), the exponents satisfy \( 1 \leq k_m < \cdots < k_2 < k_1 \), and \( \gcd(k_1, \ldots, k_m) = 1 \); or
    
    \item There exists a finite-index subgroup \( H \subset G \) preserving \( S_0 \), such that \( H \) acts trivially on the minimal projective subspace spanned by \( S_0 \).
\end{enumerate}
\end{theorem}
The dichotomy in Theorem~\ref{t:main1} echoes structural patterns in foliation theory, where analogous classifications emerge. For example, for $ d\geq 2$, the space of degree-$d$ holomorphic foliations on $\mathbb{CP}^2$ without invariant algebraic curves is residual~\cite{J}. When strengthened to non-virtually discrete cyclic groups, our main result yields a sharp classification:

\begin{corollary} \label{t:main2}
Let $G \subset \mathrm{PGL}(n+1,\mathbb{C})$, $n \geq 2$, be an infinite discrete non-virtually cyclic subgroup. Then, any $G$-invariant algebraic curve $S \subset \mathbb{CP}^n$ is either a line or a non-singular rational curve.
\end{corollary}

To prove these results, our approach combines the projection method with classical invariants of algebraic curves, such as genus, Plücker formulae, and degree, departing from the techniques in~\cite{GR, P} while building upon their insights.
The study of invariant algebraic varieties remains a central theme in modern geometry, with deep connections to dynamical systems, representation theory, and complex geometry (see \cite {GR}). These open problems not only inspire new methodologies, but also challenge established paradigms in algebraic dynamics, as discussed in~\cite{GR}.

The paper is organized as follows. Section~\ref{s:nb} introduces projective geometry conventions, curve invariants, and basic properties of monomial curves. Section~\ref{s:gec} establishes the rationality of invariant curves via projections and normalization. Section~\ref{s:par} classifies their parametrizations as monomial curves, while Section~\ref{s:examples} characterizes their automorphism groups. Key tools include Hurwitz’s theorem, Plücker formulae, and duality arguments.

\section{Preliminaries} \label{s:nb}
\subsection{Projective Geometry}\label{subsec:projective-geometry}

The \textbf{complex projective space} $\mathbb{CP}^n$ is the quotient space
\[
\mathbb{CP}^n = \big(\mathbb{C}^{n+1} - \{\mathbf{0}\}\big)/\mathbb{C}^*,
\]
where $\mathbb{C}^* = \mathbb{C} - \{0\}$ acts by scalar multiplication. We denote by 
\[
[\cdot] \colon \mathbb{C}^{n+1} - \{\mathbf{0}\} \to \mathbb{CP}^n
\] 
the canonical projection. A subset $\ell \subseteq \mathbb{CP}^n$ is a \textbf{$k$-dimensional projective subspace} if its preimage $[\ell]^{-1} \cup \{\mathbf{0}\}$ is a $(k+1)$-dimensional linear subspace of $\mathbb{C}^{n+1}$. As customary, one-dimensional spaces will be referred to as \textbf{projective lines}. For any subset $K \subseteq \mathbb{CP}^n$, we write $\langle\!\langle K \rangle\!\rangle$ for the smallest projective subspace containing $K$.

To systematically study the geometry of subspaces, we work with the \textbf{Grassmanian} $\operatorname{Gr}_k(\mathbb{CP}^n)$, defined as the space of all $k$-dimensional projective subspaces of $\mathbb{CP}^n$, endowed with the natural topology of Hausdorff convergence.

Central to our discussion are automorphisms of $\mathbb{CP}^n$, {\it i.e.} the group of \textbf{projective transformations}, which is defined as
\[
\operatorname{PGL}(n+1,\mathbb{C}) = \operatorname{GL}(n+1,\mathbb{C})/\mathbb{C}^*,
\]
where $\mathbb{C}^*$ acts by scalar multiplication. We again denote by 
\[
[\cdot] \colon \operatorname{GL}(n+1,\mathbb{C}) \to \operatorname{PGL}(n+1,\mathbb{C}).
\]
The quotient map. For $g \in \operatorname{PGL}(n+1,\mathbb{C})$, any matrix $\mathbf{g} \in \operatorname{GL}(n+1,\mathbb{C})$ with $[\mathbf{g}] = g$ is called a \textbf{lift} of $g$. Following \cite{CLU}, we say that $g$ is \textbf{elliptic} if it admits a diagonalizable lift $\mathbf{g}$ with all eigenvalues unitary. A key tool throughout this text is the projection from a fixed point, which is  defined as follows:

\begin{definition}[Projection and Induced Homomorphism, {\it cf. } \cite{CNS}]
Let $p \in \mathbb{CP}^n$ be a point, and $\mathcal{L} \subset \mathbb{CP} ^n$ a hyperplane not containing $p$. We define:
\begin{enumerate}
    \item The \emph{projection from $p$ to $\mathcal{L}$} as the map,
    \[
    \pi_{p,\mathcal{L}}: \mathbb{CP}^n - \{p\} \to \mathcal{L}, \quad x \mapsto \langle\langle x, p\rangle\rangle \cap \mathcal{L},
    \]
    where $\langle\langle x, p\rangle\rangle$ denotes the projective line through $x$ and $p$.

    \item The \emph{induced homomorphism} 
    \[
    \Pi_{p,\mathcal{L}}: G_p \to \operatorname{Bihol}(\mathcal{L}), \quad g \mapsto \left(x \mapsto \pi_{p,\mathcal{L}}(g(x))\right),
    \]
    where $G_p \leq \operatorname{PGL}(n+1,\mathbb{C})$ is the stabilizer subgroup of $p$.
    \item For every $g\in G_p$ the following diagram commutes:
\begin{equation}\label{d:proy}
\begin{tikzcd}
\mathbb{CP}^n - \{p\} \arrow[r, "g"] \arrow[d, "\pi_{p,\mathcal{L}}"'] & \mathbb{CP}^n - \{p\} \arrow[d, "\pi_{p,\mathcal{L}}"] \\
\mathcal{L} 
\arrow[r, "\Pi_{p,\mathcal{L}}(g)"] & \mathcal{L}
\end{tikzcd}
\end{equation}
\end{enumerate}
When the point $p$ and hyperplane $\mathcal{L}$ are clear from context, we simplify notation by writing $\pi = \pi_{p,\mathcal{L}}$ and $\Pi = \Pi_{p,\mathcal{L}}$.
\end{definition}

The following result highlights the interplay between invariant subspaces and point orbits for non-elliptic transformations, which will later be crucial in our analysis of curve stabilizers:
\begin{proposition}[See Lemmas 2.2 and 2.8 in \cite{CLU}]\label{prop:non-elliptic}
Let $g \in \operatorname{PGL}(n+1,\mathbb{C})$ be a non-elliptic projective transformation. Then there exist proper, non-empty invariant subspaces $\mathcal{K}, \mathcal{L} \subsetneq \mathbb{CP}^n$ (possibly equal) such that for every $q \notin \mathcal{K} $, the forward orbit $\{ g^m \cdot q:m\in \mathbb{N}\}$  is infinite, and its accumulation set lies in $\mathcal{L} $.
\end{proposition}

\subsection{Algebraic Curves}\label{subsec:algebraic-curves}

In this subsection, we introduce the fundamental notions and results concerning algebraic curves, referring to \cite{BK, GH, M} for comprehensive treatments. Throughout this section, let \( S \) be a compact Riemann surface of genus \( g \), and let \( f \colon S \to \mathbb{CP}^n \) be a holomorphic map whose image \( C = f(S) \) is reduced and non-degenerate ({\it i.e.}, \( f \) is non-constant and \( C \) is not contained in any hyperplane). Recall, also that $C$ is called {\bf totally unramified} if $C$ does not posse nodes and its inflection set is empty. 

To analyze the geometry of such curves, we first consider their \emph{associated curves}, which encode local osculating behavior and provide a framework for studying global properties.  

\begin{definition}[Associated Curves]\label{d:ac}
For each \( 0 \leq k \leq n-1 \), the \textbf{\( k \)-th associated curve} 
\[
f_k \colon S \to \operatorname{Gr}_k(\mathbb{CP}^n) \hookrightarrow \mathbb{P}\left(\bigwedge\nolimits^{k+1} \mathbb{C}^{n+1}\right)
\]
maps a point \( p \in S \) to the \textbf{osculating \( k \)-plane} at \( p \), defined as the projective span of the derivatives, {\it i.e.} \( F(p) \wedge F'(p) \wedge \cdots \wedge F^{(k)}(p) \). Here, \( F \colon U \to \mathbb{C}^{n+1}- \{0\} \) is a local holomorphic lift of \( f \) on a neighborhood \( U \subset S \) of \( p \), and derivatives are taken with respect to a local parameter. 
\end{definition}

As customary, the \textbf{degree} \( r_k (C)\) of \( f_k \) is the intersection number of \( f_k(S) \) with a generic hyperplane in \( \mathbb{P}\left(\bigwedge^{k+1} \mathbb{C}^{n+1}\right) \), counted with multiplicities. Local singularities and inflection points are captured by finer invariants called \emph{ramification indices}, which are defined below.

\begin{definition}[Ramification Indices]\label{def:ramification}
Let \( C \subset \mathbb{CP}^n \) be a curve with normalization \( f \colon S \to C \). For a point \( p \in S \) and \( 1 \leq k \leq n \), the \textbf{\( k \)-th ramification index} \( s_k^p(C) \) is defined in either of the following equivalent ways:
\begin{enumerate}
    \item \textit{Via associated curves:}  
          Consider the \( k \)-th associated curve \( f_k \colon S \to \operatorname{Gr}_k(\mathbb{CP}^n) \). Then \( s_k^p(C) \) is the ramification index of \( f_k \) at the point \( p \).

    \item \textit{Via local Laurent expansion:}  
          Choose local coordinates centered at \( p \in S \) and \( f(p) \in C \), and let \( F \colon U \to \mathbb{C}^{n+1} - \{0\} \) be a holomorphic lift of \( f \) near \( p \). Then \( F \) admits a Laurent expansion of the form
          \[
              F(t) = \big(c_0 t^{\alpha_0} + O(t^{\alpha_0+1}),\, c_1 t^{\alpha_1} + O(t^{\alpha_1+1}),\, \ldots,\, c_n t^{\alpha_n} + O(t^{\alpha_n+1})\big),
          \]
          where \( 0 = \alpha_0 < \alpha_1 < \cdots < \alpha_n \). The ramification index is then given by
          \[
              s_k^p(C) := \alpha_{k+1} - \alpha_k - 1.
          \]
\end{enumerate}
\end{definition}

A point \( p \in S \) is called a \textbf{Weierstrass point} (or {\bf W-point}) if \( s_k^p(C) \neq 0 \) for some \( k \); its image \( f(p) \) is an \textbf{inflection point} of \( C \). The curve \( C \) has \textbf{simple Weierstrass points} if every Weierstrass point \( p \) satisfies \( s_l^p(C) = 1 \) for exactly one index \( l \) and \( s_j^p(C) = 0 \) for all \( j \neq l \). The \textbf{global \( k \)-th ramification index} is defined as the sum:
\[
s_k(C) := \sum_{p \in S} s_k^p(C).
\]
We omit the notation \( (C) \) when the context is unambiguous.
 
Combining these concepts, we obtain a generalization of the classical Plücker formulas, which constrain the possible configurations, as stated below:

\begin{theorem}[Linear Plücker formulae, {\cite[p. 273]{GH}}]\label{thm:plucker}  
Let \( f \colon S \to \mathbb{CP}^n \) be a non-degenerate holomorphic map with reduced image \( C \). Then, for \( 0 \leq k \leq n-1 \),  
\[
r_{k-1} - 2r_k + r_{k+1} = 2g - 2 - s_k,  
\]  
with the conventions \( r_{-1} = r_n = 0 \).  
\end{theorem}

The following formulae, which are  derived from the Linear Plücker formulae, will also be useful:
\begin{corollary} [ See page 3 in \cite{W} ]
Let \( f \colon S \to \mathbb{CP}^n \) be a non-degenerate holomorphic map with reduced image \( C \). Then,
    \begin{equation}\label{e:plud}
        r_j = (j + 1)r_0 - 2\binom{j+1}{2} - \sum^{j-1}_{i=0} (j - i)s_i, \quad \text{for } 1 \leq j \leq n-1.
    \end{equation}
\end{corollary}

\begin{corollary} [ Plücker identity, See Lemma~1.2 in \cite{W} ]\label{l:1.2}
Let \( f \colon S \to \mathbb{CP}^n \) be a non-degenerate holomorphic map with a reduced image \( C \). Then,
    \[
        \sum_{j=0}^{n-1} s_j = r_0 + r_{n-1} - 2n.
    \]
\end{corollary}
Next, to ground these notions, we consider monomial curves.

\begin{definition}[Monomial Curves] \label{e:general} 
Let \(k_1,\ldots,k_{n} \in \mathbb{N}\) be natural numbers that satisfy \(k_1 > \cdots > k_{n} > 0\). The map \(\xi_{\mathbf{k}}:\mathbb{CP}^1 \rightarrow \mathbb{CP}^{n}\) defined by
\[
\xi_{\mathbf{k}}([z,w]) = \left[z^{k_1}: z^{k_2}w^{k_1-k_2}:\ldots: z^{k_n}w^{k_1-k_n}: w^{k_1}\right],
\]
where \(\mathbf{k}=(k_1,\ldots,k_n)\), is called a \textbf{monomial curve} and its image \textbf{monomial set}. When \(\gcd(k_1,\ldots,k_n)=1\), we call \(\xi_{\mathbf{k}}\) a \textbf{proper monomial curve}.
\end{definition}

\begin{figure}[ht]
    \centering
    \includegraphics[width=0.8\textwidth]{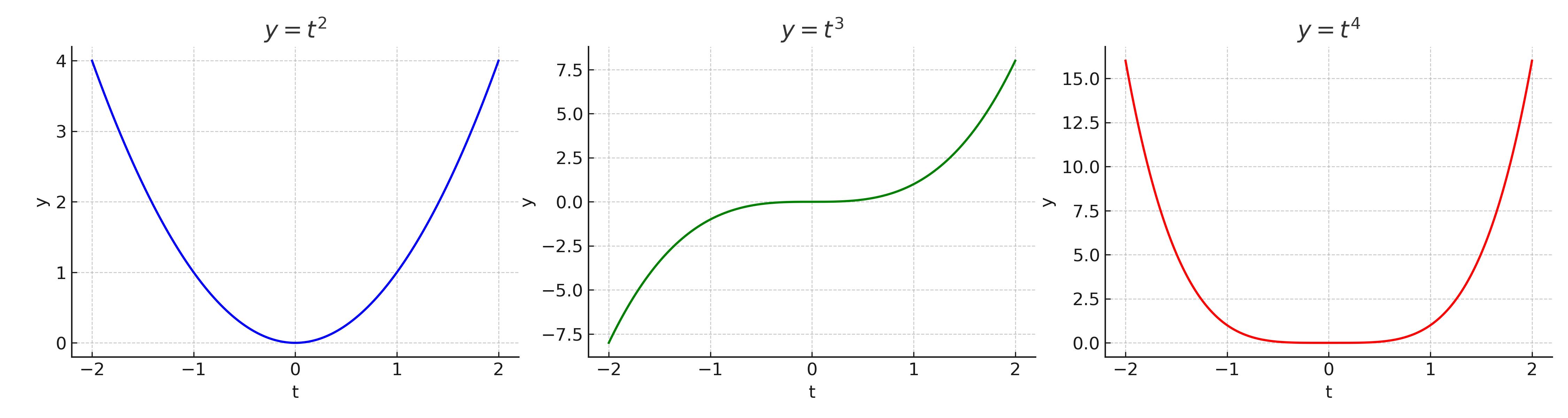}
    \caption{Examples of monomial sets}
    \label{fig:ejemplo}
\end{figure}

A canonical example of a monomial curve is the \textbf{rational normal curve}, corresponding to \(\mathbf{n}=(n,n-1,\ldots,1)\). This curve is distinguished by its total lack of Weierstrass points, making it totally unramified, and its projective automorphism group is given by:  

\begin{equation}\label{e:alex}
\iota_n\begin{pmatrix}a & b\\ c & d\end{pmatrix}_{m,j} = \frac{\binom{n}{j-1}}{\binom{n}{m-1}} \sum_{k=\delta_{m,j}}^{\Delta_{m,j}} \binom{n+1-j}{k} \binom{j-1}{m-1-k} a^{n+1-j-k}b^{j-m+k}c^k d^{m-1-k},
\end{equation}
where \(\delta_{m,j} = \max\{0,m-j\}\) and
\[
\Delta_{m,j} = \begin{cases}
m-1 & \text{if } m \leq \min\{j,n+2-j\} \text{ or } j < m \leq n+2-j,\\
n-j & \text{if } n+2-j < m \leq j \text{ or } \max\{j,n+2-j\} < m.
\end{cases}
\]
It can also be seen that this group satisfies the following relation
\begin{equation}\label{d:ver}
\begin{tikzcd}
\mathbb{CP}^1  \arrow[r, "g"] \arrow[d, "\xi_{\mathbf{n}}"'] & \mathbb{CP}^1  \arrow[d, "\xi_{\mathbf{n}}"] \\
\mathbb{CP}^n
\arrow[r, "\iota_n(g)"] & \mathbb{CP}^n
\end{tikzcd}
\end{equation}

For proper monomial curves besides the rational normal curve, the geometry is more constrained:  

\begin{itemize}
    \item The Weierstrass points are precisely \(\{[1:0], [0:1]\}\);
    \item The curve is invariant under the action of the diagonal subgroup
    \[
    \mathcal{H}_{\mathbf{k}} = \left\{\left[\operatorname{diag}\left(\alpha^{k_1}, \alpha^{k_2}\beta^{k_1-k_2},\ldots,\alpha^{k_n}\beta^{k_1-k_n},\beta^{k_1}\right)\right ] \mid \alpha,\beta \in \mathbb{C}^*\right\}.
    \]
\end{itemize}

\section{Geometry of exceptional curves} \label{s:gec}

The projection method in classical algebraic geometry, developed by 19\textsuperscript{th}-century mathematicians such as Cremona and Noether, simplifies the study of complex curves by projecting them from a carefully chosen center---one that avoids secants, tangents, and special loci---into lower-dimensional spaces, while preserving key invariants such as degree, genus, and singularities. This technique transforms high-dimensional problems into more tractable ones, enabling both curve classification and birational analysis (see \cite{GH, H, W, W1}). In this article, we build upon these classical ideas to establish our results. However, applying this method in our setting typically necessitates projecting from a singular point---a situation that, as Wall shows in \cite{W1}, can produce pathological behavior in the image curve. In the sections that follow, we demonstrate that the presence of a group preserving the curve prevents such pathologies, even when the projection originates from a singularity.

For the reader’s convenience, we begin with a brief roadmap outlining the structure and goals of this section. Our primary objective is to show that any irreducible, non-degenerate curve \( S\subset \mathbb{CP}^n \), invariant under an infinite-order projective transformation, must normalize to \( \mathbb{CP}^1 \). In Lemma \ref{l:lcontrol}, we project \( S \) from a fixed point and invoke Chow’s and Remmert’s theorems to ensure that the projected curve remains algebraic and non-degenerate. Lemma~\ref{l:fp} establishes the existence of periodic points under non-elliptic transformations, and Lemma~\ref{l:genus} combines these insights to prove that the normalization of \( S \) has genus zero. With rationality established, Corollaries \ref{c:lift}--\ref{c:finin} then show that the group action on \( S \) semiconjugates to a Möbius transformation on \( \mathbb{CP}^1 \). These results lay the foundation for the monomial parametrization developed in Section~\ref{s:par}.

\begin{lemma}\label{l:lcontrol}
    Let $S \subset \mathbb{CP}^n$ (with $n \geq 2$) be an irreducible non-degenerate algebraic curve, $p \in \mathbb{CP}^n$ a point, $\mathcal{L} \subset \mathbb{CP}^n$ a hyperplane not containing $p$, and $g \in \operatorname{PGL}(n+1,\mathbb{C})$ an infinite-order projective transformation fixing $p$ and preserving $S$. Then:
    \begin{enumerate}
        \item  The set $S_{p, \mathcal{L}}:= \overline{\pi_{p,\mathcal{L}}(S - \{p\})}$ is a non-degenerate irreducible algebraic curve in $\mathcal{L}$.
        \item The curve  $S_{p, \mathcal{L}}$ is invariant under the induced transformation $\Pi(g) \in \operatorname{PGL}(n,\mathbb{C})$, and $\Pi(g)$ has infinite order.
    \end{enumerate}
\end{lemma}

\begin{proof}
    Let $\widetilde{S}$ be the desingularization of $S$ with birational map $\psi\colon \widetilde{S} \to S$ (see \cite[p. 621]{GH}). Then the composition
    \[
    \pi_{p,\mathcal{L}} \circ \psi|_{\widetilde{S} - \psi^{-1}(p)}
    \]
    extends holomorphically to a map $\hat{\pi}\colon \widetilde{S} \to \mathbb{CP}^{n-1}$.
By Remmert's Proper Mapping Theorem \cite[p.~65]{C}, 
the image of a proper holomorphic map between complex spaces is analytic. 
Applying this to $\hat{\pi}$, we deduce that $C = \hat{\pi}(\widetilde{S})$ 
is analytic in $\mathcal{L}$. Then, by Chow's Theorem \cite[p.~73]{C}, 
it follows that $C$ is algebraic. Now we claim:

    \textbf{Claim:} We have $C = S_{p, \mathcal{L}}$.
The Inclusion $C \subset S_{p, \mathcal{L}}$ is  immediate from the density of $\widetilde{S} - \psi^{-1}(p)$ in $\widetilde{S}$. To get  reverse inclusion $S_{p, \mathcal{L}} \subset C$,  observe that given  $x \in S_{p, \mathcal{L}}$, there are  sequences:
    \begin{itemize}
        \item $x_m \in \pi_{p,\mathcal{L}}(S - \{p\})$ with $x_m \to x$
        \item $w_m \in S - \{p\}$ with $\pi_{p,\mathcal{L}}(w_m) = x_m$
        \item $u_m \in \widetilde{S} - \psi^{-1}(p)$ with $\psi(u_m) = w_m$
    \end{itemize}
    By compactness of $\widetilde{S}$, we have $u_m \to u \in \widetilde{S}$ after subsequence. Then:
    \[
    \hat{\pi}(u) = \lim_{m \to \infty} \hat{\pi}(u_m) = \lim_{m \to \infty} \pi_{p,\mathcal{L}}(w_m) = x
    \]
    Thus $x \in C$.

To establish the  non-degeneracy, let us assume, on the contrary, that 
    $S_{p, \mathcal{L}}$ were degenerate in $\mathcal{L} \cong \mathbb{CP}^{n-1}$. Then there would exist a proper subspace $\mathcal{T} \subsetneq \mathcal{L}$ containing $S_{p, \mathcal{L}}$. This would imply:
    \[
    S \subset \langle \langle \mathcal{T}, p \rangle \rangle \subsetneq \mathbb{CP}^n
    \]
    contradicting $S$'s non-degeneracy.

The remaining statements follow directly.
\end{proof}

In the following  given an algebraic curve $S\subset \mathbb{CP}^n$, a point $p\in \mathbb{CP}^n$, and a hyperplane $\mathcal{L}\subset \mathbb{CP}^n$ not containing $p$ we will use the notation  $S_{p, \mathcal{L}}:= \overline{\pi_{p,\mathcal{L}}(S - \{p\})}$. 
The projected curve $S_{p,\mathcal{L}}$ will play a central role in analyzing invariant curves. We now leverage this structure to study periodic points and normalization.

\begin{figure}[ht]
    \centering
    \includegraphics[width=0.8\textwidth]{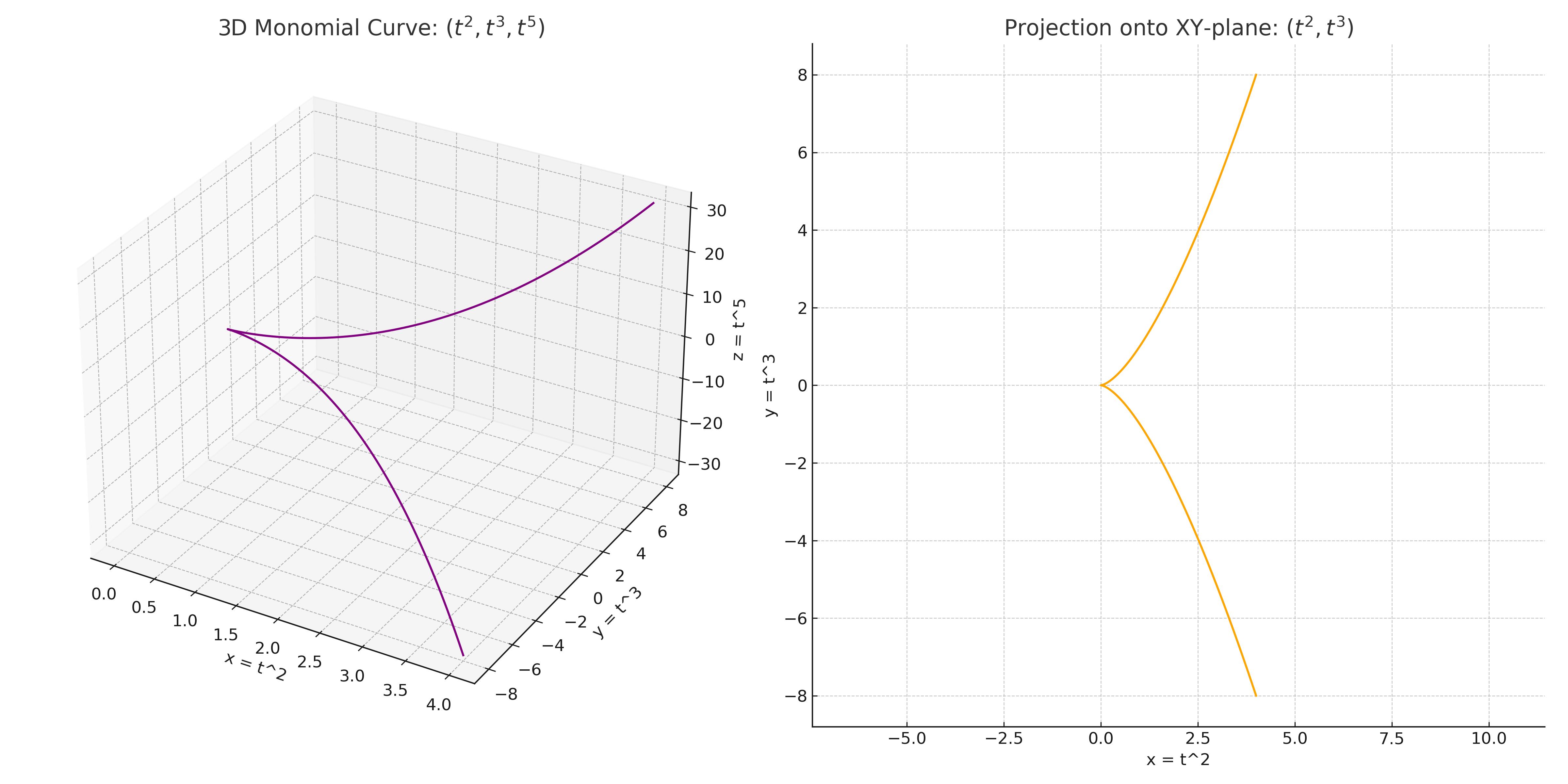}
    \caption{Example of projected monomial curve}
    \label{fig:ejemplo2}
\end{figure}

\begin{lemma} \label{l:fp}  
Let \( g \in \operatorname{PGL}(n+1,\mathbb{C}) \), with \( n \geq 2 \), be an element of infinite order, and let \( S \subset \mathbb{CP}^n \) be an irreducible, non-degenerate, \( g \)-invariant algebraic curve. Then:
\begin{enumerate}
    \item The restriction \( g|_S \) has infinite order.
    \item The curve \( S \) contains a periodic point of \( g \).
\end{enumerate}
\end{lemma}

\begin{proof}
For the elliptic case,  we may assume, after conjugation, that \(g\) admits a diagonal lift in \(\operatorname{SL}(n+1, \mathbb{C})\). Since \(S\) is irreducible, there exists a point.
\[
p \in S - \bigcup_{j=1}^{n+1} \langle\langle [e_1], \ldots, \widehat{[e_j]}, \ldots, [e_{n+1}] \rangle\rangle,
\]
where \(\widehat{[e_j]}\) denotes omission of \([e_j]\). The orbit \(\{g^m p\}_{m\in\mathbb {Z}}\) is infinite, as the diagonal action ensures all coordinates are multiplied by distinct eigenvalues.

Let \(\mathcal{L} = \langle\langle [e_1], \ldots, [e_n] \rangle\rangle\), which is a \(g\)-invariant hyperplane. By Bézout's Theorem~\cite[Theorem 18.4]{H}, the intersection \(\mathcal{L} \cap S\) is finite and non-empty. Since both \(\mathcal{L}\) and \(S\) are \(g\)-invariant, every point of \(\mathcal{L} \cap S\) is necessarily periodic under \(g\).

\noindent
In the non-elliptic case, Proposition~\ref{prop:non-elliptic} yields $g$-invariant subspaces 
\(
\mathcal{K},\;\mathcal{L}\;\subset\;\mathbb{CP}^n
\)
such that, for every $q\notin\mathcal{K}$, the forward orbit 
\(
\{\,g^m(q):m\in\mathbb{N}\}
\)
is infinite, and its accumulation set lies in $\mathcal{K}$.  Since $S$ is irreducible, choose 
\(
q_0\in S-\mathcal{L}.
\) 
Then the full orbit 
\(
\{\,g^m(q_0):m\in\mathbb{Z}\}
\)
is infinite.  Compactness and $g$-invariance of $S$ imply that at least one accumulation point of this orbit, call it $\widetilde q$, belongs to 
\(
S\cap\mathcal{L}.
\)
Finally, as $S$ and $\mathcal{L}$ are each $g$-invariant and $S$ is irreducible, the intersection 
\(
S\cap\mathcal{L}
\)
is a nonempty finite $g$-invariant set, hence composed entirely of periodic points of $g\,$.  
\end{proof}

The preceding results on invariance and periodic points constrain the geometry of 
$S$. We now exploit these constraints to classify the normalization of 
$S$.

\begin{lemma}[{\it cf.} Proposition~2 in \cite{P}] \label{l:genus}
Let \( g \in \mathrm{PGL}(n+1,\mathbb{C}) \) with \( n \geq 2 \) be an element of infinite order, and let \( S \subset \mathbb{CP}^n \) be an irreducible non-degenerate complex algebraic curve invariant under \( g \). Then the normalization \( \widetilde{S} \) of \( S \) is rational, i.e., \( \widetilde{S} \cong \mathbb{CP}^1 \).
\end{lemma}

\begin{proof}

Let \(\psi: \widetilde{S} \to S\) denote the normalization of \(S\), where \(\widetilde{S}\) is the smooth model of \(S\), and let \(\operatorname{Sing}(S)\) denote the singular locus of \(S\). Then there exists a biholomorphism
\[
\widehat{g}: \widetilde{S} - \psi^{-1}(\operatorname{Sing}(S)) \longrightarrow \widetilde{S} - \psi^{-1}(\operatorname{Sing}(S))
\]
making the diagram
\[
\begin{tikzcd}
\widetilde{S} - \psi^{-1}(\operatorname{Sing}(S)) \arrow[r, "\widehat{g}"] \arrow[d, "\psi"'] & \widetilde{S} - \psi^{-1}(\operatorname{Sing}(S)) \arrow[d, "\psi"] \\
S - \operatorname{Sing}(S) \arrow[r, "g"] & S - \operatorname{Sing}(S)
\end{tikzcd}
\]
commute. Since \(\operatorname{Sing}(S)\) is finite, by  Picard’s great theorem \(\widehat{g}\) extends holomorphically to an automorphism \(\widetilde{g} \in \operatorname{Aut}(\widetilde{S})\). By Lemma~\ref{l:fp}, \(\widetilde{g}\) has infinite order.  Hurwitz's theorem states that the automorphism group of a compact Riemann surface of genus $g \geq 2$ has order at most $84(g-1)$ \cite[Theorem 3.9]{M}. Since $\widetilde{g} \in \operatorname{Aut}(\widetilde{S})$ has infinite order, this bound forces $g = 0$ or $g = 1$.

To conclude, note that by Proposition~\ref{l:fp}, for all sufficiently large integers \(m\), the iterate \(g^m\) admits a fixed point \(x \in S\). Then
\[
x = g^m(x) = g^m\bigl(\psi(\psi^{-1}(x))\bigr) = \psi\bigl(\widetilde{g}^m(\psi^{-1}(x))\bigr),
\]
and hence \(\psi^{-1}(x)\) is a fixed point of \(\widetilde{g}^m\).
However, by Proposition~1.12 of \cite{M}, if \(\widetilde{S}\) has genus \(1\), then every point stabilizer in \(\operatorname{Aut}(\widetilde{S})\) is finite. Since \(\widetilde{g}\) has infinite order, it follows that \(\widetilde{S}\) must have genus \(0\).
\end{proof}

The rationality of 
$S$ enables a precise description of how 
$g$ lifts to the normalization, as shown next.

\begin{corollary} \label{c:lift}
Let \( n \geq 2 \), and let \( g \in \mathrm{PGL}(n+1,\mathbb{C}) \) be an element of infinite order. Suppose \( S \subset \mathbb{CP}^n \) is an irreducible, non-degenerate complex algebraic curve invariant under \( g \), and let \( \psi: \mathbb{CP}^1 \to S \) be the normalization of \( S \). Then there exists a unique Möbius transformation  \(\widetilde{g}: \mathbb{CP}^1 \to \mathbb{CP}^1\) of infinite order making the following diagram commute:
\[
\begin{tikzcd}
\mathbb{CP}^1 \arrow[r, "\widetilde{g}"] \arrow[d, "\psi"'] & \mathbb{CP}^1 \arrow[d, "\psi"] \\
S \arrow[r, "g"'] & S
\end{tikzcd}
\]
\end{corollary}
 \begin{proof}
This result immediately follows from the analysis of Lemma \ref{l:genus}, so we omit its proof for brevity.
 \end{proof}

\begin{corollary}
Let \( g \in \operatorname{PGL}(n+1,\mathbb{C}) \) (with \( n \geqslant 2 \)) be an element with infinite order, and let \( S \subset \mathbb{CP}^{n} \) be an irreducible, non-degenerate, rational algebraic complex curve that is invariant under \( g \). If \( \psi:\mathbb{CP}^{1} \to S \) is a birational equivalence, then for each \( 0 \leq k \leq n-1 \) the image of the associated curve \( \psi_{k} : S \to \operatorname{Gr}_{k}(\mathbb{CP}^{n}) \subset \mathbb{P}(\bigwedge^{k+1} \mathbb{C}^{n+1}) \), in symbols $S_k$, is invariant under \( \bigwedge^{k+1} g \).
\end{corollary}

\begin{proof}
Since $S_{k}$ is an algebraic curve, it suffices to verify the lemma for the smooth points of $S_{k}$. Let $p = \psi_{k}(z)$ be such a smooth point. By definition, $p$ is the unique $k$-plane in $\mathrm{Gr}_{k}(\mathbb{CP}^{n})$ that has contact of order at least $k+1$ with $S$ at $\psi(z)$---that is, $p$ is the $k$-osculating plane of $S$ at $\psi(z)$ (see \cite[page 264]{GH}).

The group action of $g$ on $\mathrm{Gr}_{k}(\mathbb{CP}^{n}) \subset \mathbb{P}(\bigwedge^{k+1} \mathbb{C}^{n+1})$ is given by a biholomorphism. Consequently, $g \cdot p$ is a $k$-plane that has contact of order at least $k+1$ with $S$ at $g \cdot \psi(z)$. By uniqueness, $g \cdot p$ must therefore be the $k$-osculating plane of $S$ at $g \cdot \psi(z)$, meaning $g \cdot p = \psi_{k}(g \cdot z)$.  

This shows that $g \cdot p \in S_{k}$, establishing the $g$-invariance of $S_{k}$.  
\end{proof}

\begin{proof}
Let \( \widetilde{g} \) be as in Corollary \ref{c:lift}, to conclude observe \( \psi_{k} \circ \widetilde{g} = \bigwedge^{k+1} g \circ \psi_{k} \).
\end{proof}

Due to the previous corollary, we get:

\begin{corollary}
Let \( g \in \operatorname{PGL}(n+1,\mathbb{C}) \) (with \( n \geqslant 2 \)) be an element with infinite order, and let \( S \subset \mathbb{CP}^{n} \) be an irreducible, non-degenerate, rational algebraic complex curve that is invariant under \( g \). If \( \psi:\mathbb{CP}^{1} \to S \) is a birational equivalence, then the set of inflections point is invariant under $g$
\end{corollary}

\begin{corollary}[{\textit{cf.}} Corollary 2 in \cite{P}]\label{c:finin}
Let \( g \in \operatorname{PGL}(n+1, \mathbb{C}) \) (with \( n \geq 2 \)) be an element of infinite order, \( S \subset \mathbb{CP}^n \) an irreducible, non-degenerate rational complex curve that is \( g \)-invariant, and \( \psi: \mathbb{CP}^1 \to S \) a birational map. If \( \Sigma \subset S \) is a finite \( g \)-invariant subset, then \( \psi^{-1}(\Sigma) \) contains at most two points. In particular, for a fixed point \( p \in S \) of \( g \), the following hold:
\begin{enumerate}
    \item \label{i:pi1} The set 
    \[
   \widetilde{Sing}(S) = \left\{ x \in \mathbb{CP}^1 : x \textrm{ is a  W -point of } \psi \right\} \cup \psi^{-1}\big(\operatorname{Sing}(S) \cup \{p\}\big)
    \]
    contains at most two points.
    \item  \label{i:pi2}The singular set \( \operatorname{Sing}(S) \) is either empty, consists of one cusp, two cusps, or a node with two branches.
\end{enumerate}
\end{corollary}

\begin{proof} 
Let \(\Sigma\subset S\) be a \(g\)-invariant finite set. If \( \widetilde{g} \) is the Möbius transformation of infinite order given by Corollary \ref{c:lift}, then \( \psi^{-1}(\Sigma) \) is a finite \( \widetilde{g} \)-invariant set. However, a Möbius transformation of infinite order can only preserve a finite set if that set contains at most two points. It follows that \( \psi^{-1}(\Sigma) \) contains at most two elements.

To show Item (\ref{i:pi1}), first observe that the set of  W-points is finite (see pages 264--266 in \cite{GH}). Since \( g \) is a biholomorphism, we deduce that \( g \) preserves singularities and the inflection set of \( S \). Thus, \(\widetilde{Sing}(S) \) contains at most two points.

The proof of Item (\ref{i:pi2}) is straightforward, so we omit it here.
\end{proof}

With the normalization established, we now examine how invariants of 
$g$ constrains the critical structure of projections.

\begin{lemma}\label{l:theta}
Let \( g \in \operatorname{PGL}(n+1, \mathbb{C}) \) (for \( n \geq 2 \)) be an element of infinite order, and let \( S \subset \mathbb{CP}^n \) be a non-degenerate irreducible rational complex algebraic curve that is invariant under \( g \). Suppose \( p \in S \) is a fixed point of \( g \), and let \( \mathcal{L} \subset \mathbb{CP}^n \) be a hyperplane not containing \( p \). Denote by \( \psi: \mathbb{CP}^1 \to S \) and \( \nu: \mathbb{CP}^1 \to S_{p,\mathcal{L}} \) birational maps, where \( S_{p,\mathcal{L}} \) is the projection of \( S \) from \( p \) onto \( \mathcal{L} \). 

Then there exist: 
\begin{itemize}
    \item Elements \( \widetilde{g}, \widehat{g} \in \operatorname{PGL}(2, \mathbb{C}) \) of infinite order,
    \item A rational map \( \vartheta: \mathbb{CP}^1 \to \mathbb{CP}^1 \),
\end{itemize}
such that the following hold:
\begin{enumerate}
    \item \label{i:pp1} The composition \( \nu \circ \vartheta \) coincides with \( \pi_{p,\mathcal{L}} \circ \psi \), as birational maps.

    \item For every integer \( m \in \mathbb{Z} \), the diagram below commutes:

\begin{equation}\label{e:diagram}
\begin{tikzcd}[row sep=3em, column sep=3.5em]
& \mathbb{CP}^1 \arrow[rr, "\vartheta", near start] \arrow[dd, "\widetilde{g}^m" near start, swap] \arrow[dl, "\psi" near end, swap] & & \mathbb{CP}^1 \arrow[dl, "\nu" near end] \arrow[dd, "\widehat{g}^m" near start] \\
S \arrow[rr, crossing over, "\pi_{p,\mathcal{L}}" near start] \arrow[dd, "g^m" near start, swap] & & S_{p,\mathcal{L}} \\
& \mathbb{CP}^1 \arrow[rr, "\vartheta" near start] \arrow[dl, "\psi" near end, swap] & & \mathbb{CP}^1 \arrow[dl, "\nu" near end] \\
S \arrow[rr, "\pi_{p,\mathcal{L}}" near start] & & S_{p,\mathcal{L}} \arrow[from=uu, crossing over, "\Pi_{p,\mathcal{L}}(g^m)" near end, swap]
\end{tikzcd}
\end{equation}

\end{enumerate}
\end{lemma}

\begin{proof}
Define the domain  
\[
U = \mathbb{CP}^1 - \Big(   \widetilde{Sing}(S) \cup \psi^{-1}\big(\pi^{-1}_{p,\mathcal{L}}(\operatorname{Sing}(S_{p,\mathcal{L}}))\Big),
\]  
and let \( \widetilde{\vartheta}: U \to \mathbb{CP}^1 \) be the holomorphic map given by  
\[
\widetilde{\vartheta}([z, w]) = \nu^{-1}\Big(\pi_{p,\mathcal{L}}\big(\psi([z, w])\big)\Big).
\]  
Here, \( \nu^{-1} \) is well-defined on the smooth locus of \( S_{p,\mathcal{L}} \), and \( U \) excludes the indeterminacy loci of \( \psi \) and the singularities of \( S_{p,\mathcal{L}} \).
By Bézout's Theorem, the projection \( \pi_{p,\mathcal{L}} \circ \psi \) has finite fibers. Thus, there exists \( k_0 \in \mathbb{N} \) such that every point in \( \widetilde{\vartheta}(U) \) has at most \( k_0 \) preimages under \( \widetilde{\vartheta} \).
The map \( \widetilde{\vartheta} \) omits at most finitely many points of \( \mathbb{CP}^1 \) (due to the excluded locus \( U \)). By Picard's Great Theorem, \( \widetilde{\vartheta} \) extends holomorphically to a map \( \vartheta: \mathbb{CP}^1 \to \mathbb{CP}^1 \). Since \( \vartheta \) is holomorphic and \( \mathbb{CP}^1 \) is compact, \( \vartheta \) is necessarily a rational map. By construction, \( \vartheta \) satisfies  
\[
\nu \circ \vartheta(z) = \pi_{p,\mathcal{L}} \circ \psi(z) \quad \text{for all } z \in U.
\]

Following the methodology of Lemma \ref{l:genus} (which constructs semiconjugacies for automorphisms of curves), there exist elements \( \widetilde{g}, \widehat{g} \in \operatorname{PGL}(2, \mathbb{C}) \) of infinite order such that:  
\begin{align*}
\psi \circ \widetilde{g}^m &= g^m \circ \psi, \\
\nu \circ \widehat{g}^m &= \Pi(g)^m \circ \nu \quad \text{for all } m \in \mathbb{Z}.
\end{align*}

To establish \( \widehat{g}^m \circ \vartheta = \vartheta \circ \widetilde{g}^m \), observe that both sides are rational maps coinciding on the Zariski-dense open set \( U \). By the identity theorem for rational maps, equality holds globally. This completes the commutativity of the diagram.  
\end{proof}

The semiconjugacy 
$\vartheta$ tightly controls the critical locus, as summarized below.

\begin{corollary}\label{l:theta2}
Adopting the hypotheses and notation of Lemma \ref{l:theta}, the following hold:
\begin{enumerate}
    \item \label{i:th1} The map \( \pi_{p,\mathcal{L}} \) establishes a birational equivalence precisely when \( \vartheta \) lies in \( \operatorname{PGL}(2, \mathbb{C}) \).
    \item \label{i:th3} The degree of \( \vartheta \) exceeds one if and only if \( \vartheta \) is a bicritical map.
    \item \label{i:th2} The critical set \( C_{\vartheta} \) of \( \vartheta \) is invariant under \( \widetilde{g} \).
    \item \label{i:th4} For bicritical \( \vartheta \), the critical set coincides with both \(\widetilde{Sing}(S) \) and the fixed locus, {\it i.e.}: \( C_{\vartheta} =\widetilde{Sing}(S) = \operatorname{Fix}(\widetilde{g}) \).
    \item \label{i:th5} The critical value set \( V_{\vartheta} \) of \( \vartheta \) remains invariant under \( \widehat{g} \).
    \item \label{i:th6} Let us define \(\widetilde{\operatorname{Sing}}(S_{p,\mathcal{L}}) := \nu^{-1}\bigl(\operatorname{Sing}(S_{p,\mathcal{L}})\bigr)\).  
    Then, \(\widetilde{\operatorname{Sing}}(S_{p,\mathcal{L}})\) is \(\widehat{g}\)-invariant.  
    Moreover, if \(\widetilde{\operatorname{Sing}}(S_{p,\mathcal{L}})\) consists of two points, then  
    \[
    \vartheta(\widetilde{Sing}(S)) \subseteq \widetilde{\operatorname{Sing}}(S_{p,\mathcal{L}}).
    \]
\end{enumerate}
\end{corollary}

\begin{proof}
Proof of (\ref{i:th1}): The statement follows immediately from Part (\ref{i:pp1}) of  Lemma \ref{l:theta}, so we omit the proof.

Proof of (\ref{i:th3}): By Corollary \ref{c:finin}, \( C_{\vartheta} \) contains at most two critical points. By \cite{V}, no rational map is unicritical. Hence, if \( \deg(\vartheta) > 1 \), \( \vartheta \) must admit exactly two critical points, i.e., \( \vartheta \) is bicritical.

Proof of (\ref{i:th2}): Let \( q \in C_{\vartheta} -\widetilde{Sing}(S) \), then  there exists an open neighborhood \( U \) of \( q \) such that \( \vartheta : U - \{ q \} \to \vartheta(U) - \{ \vartheta(q) \} \) is a \( k \)-sheeted covering map for some \( k \). By Lemma 4.4, the relation  
\[
\widehat{g}|_{\vartheta(U) - \{\vartheta(q)\}} \circ \vartheta|_{U - \{ q \}} = \vartheta|_{\widehat{g}(U) - \{\widehat{g}(q)\}} \circ \widetilde{g}|_{U - \{ q \}}
\]  
holds. Since \( \widetilde{g} \) and \( \widehat{g} \) are homeomorphisms and \( \vartheta|_{U - \{ q \}} \) is a covering map, it follows that \( \widetilde{g}(q) \in C_{\vartheta} \). Thus, \( C_{\vartheta} \) is \( \widetilde{g} \)-invariant.

Proof of Part (\ref{i:th4}).
Combining the results from parts (\ref{i:th3}) and (\ref{i:th2}) with the fact that $\widetilde{g}$ is a Möbius transformation of infinite order, we deduce that $C_{\vartheta} = \operatorname{Fix}(\widetilde{g})$. As $\widetilde{Sing}(S)$ contains at most two points, the proof reduces to showing the inclusion $C_{\vartheta} \subset\widetilde{Sing}(S)$.

Consider an arbitrary point $q \in C_{\vartheta} - \psi^{-1}(p)$. We may choose a sufficiently small open neighborhood $U$ of $q$ satisfying:
\begin{itemize}
    \item $\vartheta|_{U-\{q\}} \colon U-\{q\} \to \vartheta(U)-\{\vartheta(q)\}$ forms a $\deg(\vartheta)$-sheeted covering map;
    
    \item Both restricted maps $\psi|_{U-\{q\}} \colon U-\{q\} \to \psi(U)-\{\psi(q)\}$ and $\nu|_{\vartheta(U)-\{\vartheta(q)\}}$ are homeomorphisms;
    
    \item The following diagram commutes:
    \[
    \begin{tikzcd}
    U-\{q\} \arrow[r,"\vartheta"] \arrow[d,"\psi"] & \vartheta(U)-\{\vartheta(q)\} \arrow[d,"\nu"] \\
    \psi(U)-\{\psi(q)\} \arrow[r,"\pi_{p,\mathcal{L}}"] & \nu(\vartheta(U))-\{\nu(\vartheta(q))\}
    \end{tikzcd}
    \]
\end{itemize}

From these properties, we immediately conclude that $\pi_{p,\mathcal{L}}$ induces a $\deg(\vartheta)$-sheeted covering:
\[
\pi_{p,\mathcal{L}} \colon \psi(U)-\{\psi(q)\} \to \pi_{p,\mathcal{L}}(\vartheta(U))-\{\pi_{p,\mathcal{L}}(\vartheta(q))\}.
\]

For any $r \in \pi_{p,\mathcal{L}}(\vartheta(U))-\{\pi_{p,\mathcal{L}}(\vartheta(q))\}$, the line $\ell_r = \langle \langle p,r\rangle \rangle$ intersects $\psi(U)-\{\psi(q)\}$ at exactly $\deg(\vartheta)$ distinct points. This intersection multiplicity forces the osculating line to $S$ at $p$ to have contact order $\geq \deg(\vartheta) + 1$. The assumption $\deg(\vartheta) \geq 2$ guarantees that $\psi(q)$ must be a singular point of $S$, i.e., $q \in\widetilde{Sing}(S)$, completing the proof of this Part.

Proof of (\ref{i:th5}): The invariance of \( V_{\vartheta} \) under \( \widehat{g} \) follows directly from the \( \widetilde{g} \)-invariance of \( C_{\vartheta} \); details are omitted.

Proof of (\ref{i:th6}): The $\widehat{g}$-invariance of $g$ is immediate, so we omit its proof here. To establish the rest of the proof, we consider two cases:  

\textbf{Case 1:} The map \(\vartheta\) is birational.  
By Part~(\ref{i:th4}) of this lemma, we have  
\(
    \widetilde{\operatorname{Sing}}(S) = C_{\vartheta}.
\)  
Furthermore, since both \(V_{\vartheta}\) and \(\widetilde{\operatorname{Sing}}(S_{p,\mathcal{L}})\) are \(\widehat{g}\)-invariant, it follows that  
\(
    V_{\vartheta} = \widetilde{\operatorname{Sing}}(S_{p,\mathcal{L}}).
\)  
Applying Part~(\ref{i:th4}) again, this equality implies  
\[
    \vartheta\bigl(\widetilde{\operatorname{Sing}}(S)\bigr) = \widetilde{\operatorname{Sing}}(S_{p,\mathcal{L}}).
\] 

\textbf{Case 2:} Suppose \(\vartheta \in \operatorname{PGL}(2, \mathbb{C})\). 
By Lemma~\ref{l:theta}, we have the conjugation relation
\(
    \widehat{g}^m \circ \vartheta = \vartheta \circ \widetilde{g}^m.
\)
For any fixed point \(x \in \operatorname{Fix}(\widetilde{g}^m)\), this implies
\(
    \widehat{g}^m(\vartheta(x)) = \vartheta(x),
\)
and consequently \(\vartheta(x) \in \widetilde{\operatorname{Sing}}(S_{p,\mathcal{L}})\). 
Since \(\widetilde{\operatorname{Sing}}(S) \subset \operatorname{Fix}(\widetilde{g})\) by assumption, we obtain the inclusion,
\[
    \vartheta\bigl(\widetilde{\operatorname{Sing}}(S)\bigr) \subset \widetilde{\operatorname{Sing}}(S_{p,\mathcal{L}}).
\]
This completes the proof of this Part.
 \end{proof}

            \section{Parametrization of the exceptional set} \label{s:par}

Having shown that any invariant curve normalizes to \(\mathbb{CP}^1\), we now prove that every invariant rational curve is projectively equivalent to a monomial set. To achieve this,  in Corollaries \ref{c:node}–\ref{c:inflection2} we rule out nodes and force exactly 0 or 2 inflection points; Finally, Lemma \ref{l:cc} and the concluding arguments at the end establish Theorem \ref{t:main1}.

Central to our argument is the following bridging result between geometric invariance and algebraic structure:

\begin{lemma} \label{l:para}  
Let \( g \in \mathrm{PGL}(n+1,\mathbb{C}) \) (for \( n \geq 2 \)) be an element of infinite order, \( S \subset \mathbb{CP}^n \) a non-degenerate irreducible rational curve invariant under \( g \), \( p \in S \) a fixed point of \( g \), and \( \mathcal{L} \subset \mathbb{CP}^n \) a hyperplane avoiding \( p \). Suppose \( S_{p,\mathcal{L}} \) is projectively equivalent to a monomial set. Then there exist:  
\begin{itemize}
    \item A strictly decreasing sequence \( k_1 > \cdots > k_{n-1} \) of natural numbers,
    \item A desingularization \( \psi: \mathbb{CP}^1 \to S \),
    \item Transformations \( m \in \mathrm{PGL}(2,\mathbb{C}) \) and \( h \in \mathrm{PGL}(n+1,\mathbb{C}) \),
\end{itemize}
such that:  
\begin{enumerate}
    \item \( h(p) = [e_{n+1}] \),
    \item \( h(\mathcal{L}) = \langle\langle [e_{1}],\ldots, [e_n] \rangle\rangle \),
    \item \( m^{-1}(\psi^{-1}(p)) \subset \{[0:1], [1:0]\} \),
    \item The following equality holds:  
    \[
    \pi_{[e_{n+1}], \langle \langle [e_1], \ldots, [e_n] \rangle \rangle} \circ h \circ \psi \circ m ([z:w]) =  \left[z^{k_1}: z^{k_2}w^{k_1 - k_2}: \ldots : z^{k_n}w^{k_1 - k_n}: w^{k_1}: 0\right].
    \]
\end{enumerate}  

If \( q \in S - \{p\} \) is another fixed point of \( g \), the same \( \psi, h, m \) may be chosen to additionally satisfy \( h(q) = [e_1] \) and \( m^{-1}(\psi^{-1}(\{p,q\})) \subset \{[0:1], [1:0]\} \).  
\end{lemma}			

\begin{proof} 
If \( S \) is totally unramified, the result follows immediately. Assume instead that \( S \) is not totally unramified. Let \( \psi: \mathbb{CP}^1 \to S \) and \( \nu: \mathbb{CP}^1 \to S_{p,\mathcal{L}} \) be desingularizations. By Lemma \ref{l:theta}, there exist maps \( \vartheta \), \( \widetilde{g} \), and \( \widehat{g} \). After conjugation, we may assume:  
 \( p = [e_{n+1}] \),
 \( \mathcal{L} = \langle \langle [e_1], \ldots, [e_n] \rangle \rangle \),
  a tuple \( \mathbf{k} = (k_1, \ldots, k_n) \in \mathbb{Z}^n \) with \( k_1 > \cdots > k_n \geq 1 \) and \( \gcd(k_1, \ldots, k_n) = 1 \), and such that \[ \nu([z:w]) = \left[z^{k_1}: \cdots : z^{k_n}w^{k_1 - k_n}: w^{k_1}: 0\right]. \] 

We divide the proof into four cases based on the properties of \( S_{p,\mathcal{L}} \) and \( \vartheta \):

\medskip
\noindent\textbf{Case 1:} \( S_{p,\mathcal{L}} \) is non-totally unramified and \( \vartheta \in \mathrm{PSL}(2,\mathbb{C}) \).  
Since \( S_{p,\mathcal{L}} \) is non-totally unramified, \( \widetilde{\mathrm{Sing}}(S_{p,\mathcal{L}}) = \{[0:1], [1:0]\} \). Choose \( m \in \mathrm{PGL}(2,\mathbb{C}) \) such that:  
\[
m[1:0] = \vartheta^{-1}([1:0]), \quad m[0:1] = \vartheta^{-1}([0:1]).
\]
This implies \( \vartheta(m([z:w])) = [az:bw] \) for \( a, b \in \mathbb{C}^* \). Define the diagonal matrix:  
\[
h_1 = \mathrm{Diag}\left(a^{-k_1},\, a^{-k_2}b^{k_2 - k_1},\, \ldots,\, a^{-k_n}b^{k_n - k_1},\, b^{-k_1},\, 1\right).
\]
By Part (\ref{i:th6})  of Corollary \ref{l:theta2}, we know \( \vartheta(\widetilde{\operatorname{Sing}}(S)) \subset \widetilde{\operatorname{Sing}}(S_{p,\mathcal{L}}) \), so we deduce  \( \psi^{-1}(p) \subset \vartheta^{-1}\{[1:0],[0:1]\} \).  Thus, for each  \([z:w] \in \mathbb{CP}^1-\{[1:0],[0:1]\}\), we get:  
\begin{footnotesize}
\[
\begin{array}{ll}
(\pi_{p,\mathcal{L}} \circ h_1  \circ \psi \circ m)([z:w])
&=(\pi_{p,\mathcal{L}} \circ h_1  \circ \psi) (\vartheta^{-1}( [az:bw]))\\
&=\pi_{p,\mathcal{L}} ( h_1(   \pi_{p,\mathcal{L}}^{-1}( \nu ([az:bw]))))\\
&=\pi_{p,\mathcal{L}} ( h_1(   \pi_{p,\mathcal{L}}^{-1}(\left[(az)^{k_1}: \ldots:(az)^{k_n}(bw)^{k_1 - k_n}  :(b w)^{k_1}: 0\right]))))\\
&=\pi_{p,\mathcal{L}} ( h_1(  \left[(az)^{k_1}: \ldots:(az)^{k_n}(bw)^{k_1 - k_n}  :(b w)^{k_1}: r\right]))\\
&=\pi_{p,\mathcal{L}} (  \left[z^{k_1}: z^{k_2}w^{k_1 - k_2}: \ldots : z^{k_n}w^{k_1 - k_n}: w^{k_1}: r\right])\\
&=  \left[z^{k_1}: z^{k_2}w^{k_1 - k_2}: \ldots : z^{k_n}w^{k_1 - k_n}: w^{k_1}: 0\right] \\
\end{array}
\]
\end{footnotesize}

\medskip
\noindent\textbf{Case 2:} \( S_{p,\mathcal{L}} \) is non-totally unramified and \( \vartheta \) is bicritical.   As before, we have 
 \( \widetilde{\mathrm{Sing}}(S_{p,\mathcal{L}}) = \{[0:1], [1:0]\} \). By  Parts (\ref{i:th5}) and  (\ref{i:th6}) of corollary \ref{l:theta2}, we get that \( \vartheta(\widetilde{\operatorname{Sing}}(S)) \subset \widetilde{\operatorname{Sing}}(S_{p,\mathcal{L}}) \) and \( V_\vartheta = \widetilde{\operatorname{Sing}}(S_{p,\mathcal{L}}) \). Choosing \( m \) as in Case 1, we find \( \vartheta \circ m \) has critical points at \([1:0], [0:1]\), so \( \vartheta(m([z:w])) = [az^l:bw^l] \) for some \( l > 1 \). Define \( h_1 \) analogously to Case 1,  so for $[z,w]\in \mathbb{CP}^1- \{[0:1], [1:0]\}$ using similar arguments to the previous case we get:  
\[
\pi_{p,\mathcal{L}} \circ h_1 \circ \psi \circ m ([z:w]) = \left[z^{k_1l}: \cdots : z^{k_nl}w^{k_1l - k_nl}: w^{k_1l}: 0\right].
\]

\medskip
\noindent\textbf{Case 3:} \( S_{p,\mathcal{L}} \) is totally unramified and \( \vartheta \in \mathrm{PSL}(2,\mathbb{C}) \).
Here, \( \mathbf{k} = (n-1, \ldots, 1) \). By Diagram \ref{e:diagram}, \( \mathrm{Card}(\mathrm{Fix}(\widehat{g})) = \mathrm{Card}(\mathrm{Fix}(\widetilde{g})) \). Choose \( m_1, m_2 \in \mathrm{PGL}(2,\mathbb{C}) \) such that \( m_1(\mathrm{Fix}(\widehat{g})) = m_2^{-1}(\mathrm{Fix}(\widetilde{g})) \subset \{[0:1], [1:0]\} \). The composition \( m_1 \circ \vartheta \circ m_2 \) takes one of four forms: $[az: bw]$, 
$[bw: az]$,  $[az+ bw: cw]$, $[az: bz + cw]$; by prior reduction, we need only address \( m_1(\vartheta(m_2([z:w]))) = [az + bw:cw] \). Define:  
\[
h_1 = \begin{bmatrix} \mathbf{h} & 0 \\ 0 & 1  \end{bmatrix},
\]
where \( \mathbf{h} \in \mathrm{SL}(n,\mathbb{C}) \) lifts \( \iota_{n-1}\left(\left(\begin{smallmatrix} a & b \\ 0 & c \end{smallmatrix}\right)^{-1} m_1\right) \).  Once again in this case $\psi(-1)(p)\in \{[1:0],[0:1]\}$, so  for each $[z,w]\in \mathbb{CP}^1-\{[1:0],[0:1]\}$ we derive:  
\[
\begin{array}{ll}
(\pi_{p,\mathcal{L}} \circ h_1  \circ \psi \circ m_2)([z:w])
&=\pi_{p,\mathcal{L}} ( h_1  ( \psi( \vartheta^{-1}(m_1^{-1}( [az+ bw: cw])))))\\
&= \Pi_{p,\mathcal{L}}(h_1) ( \nu (m_1^{-1}( [az+ bw: cw])))\\
&=\left [\iota_{n-1}\left(\left(\begin{smallmatrix} a & b \\ 0 & c \end{smallmatrix}\right)^{-1} m_1\right)(\xi_{\mathbf{n}} (m_1^{-1}( (az+ bw, cw)))):0\right ]\\
&= \left [\xi_{\mathbf{n}}(\left(\begin{smallmatrix} a & b \\ 0 & c \end{smallmatrix}\right)^{-1} m_1m_1^{-1}( (az+ bw: cw)))):0\right ]\\
&=
\left[z^{n}: z^{n-1}w^{1}: \ldots : zw^{n-1}: w^{n}: 0\right] \\
\end{array}
\]

\medskip
\noindent\textbf{Case 4:} \( S_{p,\mathcal{L}} \) is totally unramified and \( \vartheta \) is bicritical.
Similar to Case 3, we have \( \mathbf{k} = (n-1, \ldots, 1) \). Here, \( \mathrm{Fix}(\widehat{g}) = V_\vartheta \) and \( \mathrm{Fix}(\widetilde{g}) = C_\vartheta \). Choosing \( m_1, m_2 \) such that \( m_1(\mathrm{Fix}(\widehat{g})) = m_2^{-1}(\mathrm{Fix}(\widetilde{g})) = \{[1:0], [0:1]\} \), we find \( m_1 \circ \vartheta \circ m_2([z:w]) = [az^l:bw^l] \). Defining \( h_1 \) 
\[
h_1 = \begin{bmatrix} \mathbf{h} & 0 \\ 0 & 1  \end{bmatrix},
\]
where \( \mathbf{h} \in \mathrm{SL}(n,\mathbb{C}) \) lifts \( \iota_{n-1}\left(\left(\begin{smallmatrix} a & 0 \\ 0 & b \end{smallmatrix}\right)^{-1} m_1\right) \). As in the previous case,  a straightforward computation shows that for each $[z,w]\in\mathbb{CP}^1-\{[1:0],[0:1]\}$  we have:

\[
\pi_{p,\mathcal{L}} \circ h_1 \circ \psi \circ m_2 ([z:w]) = \left[z^{ln}: \cdots : w^{ln}: 0\right].
\]

\medskip
\noindent In all cases, the required transformations \( \psi, h, m \) exist. If \( q \in S - \{p\} \) is another fixed point, analogous choices ensure \( h(q) = [e_1] \) and \( m^{-1}(\psi^{-1}(\{p,q\})) \subset \{[0:1], [1:0]\} \).
\end{proof}

With this parametrization machinery in place, we first eliminate nodal singularities:

\begin{corollary}[Nodal Exclusion] \label{c:node}
Let \( g \in \mathrm{PGL}(n+1,\mathbb{C}) \) with \( n \geq 2 \) be an element of infinite order, and let \( S \subset \mathbb{CP}^n \) be a non-degenerate, irreducible rational algebraic curve invariant under \( g \). Suppose \( p \in S \) is a fixed point of \( g \), and let \( \mathcal{L} \subset \mathbb{CP}^n \) be a hyperplane not containing \( p \). If the projection \( S_{p,\mathcal{L}} \) is projectively equivalent to a monomial set, then \( S \)  cannot have a node.
\end{corollary}

\begin{proof}
On the contrary, let us assume that $S$ contains a node; thus, by Corollary \ref{c:finin}, this should be the fixed point $p$.   On the other hand,   by Lemma~\ref{l:para}, we may assume that  \( p = [e_{n+1}] \) and $\mathcal{L}=\langle\langle [e_1],\ldots[e_n]\rangle\rangle$. The lemma further guarantees the existence of homogeneous polynomials \( P(z,w), Q(z,w) \in \mathbb{C}[z,w] \), a sequence of natural numbers \( k_1 > k_2 > \ldots > k_{n-1} \), and a desingularization \( \psi: \mathbb{CP}^1 \to S \) satisfying \( \psi^{-1}(p) = \{[0,1], [1,0]\} \) and
\begin{scriptsize}
    \[
        \psi[z,w] = \left[ z^{k_1}P(z,w):\, z^{k_2}w^{k_1 - k_2}P(z,w):\, \ldots:\, z^{k_{n-1}}w^{k_1 - k_{n-1}}P(z,w):\, w^{k_1}P(z,w):\, Q(z,w) \right].
    \]
    \end{scriptsize}
Observe the invariance of the set  \( \widetilde{Sing}(S) \) and \( \psi^{-1}(p) = \{[0,1], [1,0]\} \) forces  \( P(z,w) = az^kw^l \) for constants \( a \in \mathbb{C}^* \) and \( k, l \in \mathbb{N} \cup \{0\} \).  

    To analyze \( \psi \) locally, consider the standard affine charts \( U_z = \{[z:1]\} \) and \( U_w = \{[1:w]\} \) on \( \mathbb{CP}^1 \). Restricting \( \psi \) to these charts yields lifted maps:
    \[
        \begin{aligned}
            A(z) &= \left( az^{k + k_1},\, az^{k + k_2},\, \ldots,\, az^{k + k_{n-1}},\, az^k,\, r(z) \right), \\
            B(w) &= \left( aw^l,\, aw^{k_1 - k_2 + l},\, \ldots,\, aw^{k_1 - k_{n-1} + l},\, aw^{k_1 + l},\, s(w) \right),
        \end{aligned}
    \]
    where \( r(z) = Q(z,1) \) and \( s(w) = Q(1,w) \). The non-vanishing conditions \( r(0) \neq 0 \neq s(0) \) ensure regularity at \( z=0 \) and \( w=0 \). Thus, a  direct computation of the ramification indexes  at $[1:0]$ and $[0:1]$ is possible,  and gives:
    \[
        \begin{array}{l}
            s^{[1,0]}_0 = k - 1,\\  
            s^{[1,0]}_1 = k_{n-1} - 1, \\
            s^{[1,0]}_j = k_{n-j} - k_{n-j+1} - 1 \quad (2 \leq j \leq n-1); \\
            s^{[0,1]}_0 = l - 1, \\ 
            s^{[0,1]}_j = k_j - k_{j+1} - 1 \quad (1 \leq j \leq n-2),\\ 
            s^{[0,1]}_{n-1} = k_{n-1} - 1.
        \end{array}
    \]
    Since  \( S \) admits  a node at \( p \), then \( S - \{p\} \)  lacks of  inflection points. This imposes  the following constraints on the global ramification indexes:
    \[
        \begin{aligned}
            s_0 &= k + l - 2, \\
            s_1 &= k_{n-1} + k_1 - k_2 - 2 = s_{n-1}, \\
            s_j &= k_j - k_{j+1} + k_{n-j} - k_{n-j+1} - 2 \quad (2 \leq j \leq n-2).
        \end{aligned}
    \]

Next, we claim:

    \textbf{Claim.} \( k_1 + k + l = 0 \).  
    For \( n = 2, 3 \), this holds trivially. Assume \( n \geq 4 \). Substituting \( j = n-1 \) into Equation~\eqref{e:plud} gives:
    \begin{equation} \label{e:son}
        -(k_1 + k + l) + n^2 - 3n + n(k_{n-1} - k_2) + \sum^{n-2}_{i=2} (n - i)s_i = 0.
    \end{equation}
    Substituting \( s_i = k_i - k_{i+1} + k_{n-i} - k_{n-i+1} - 2 \) into the summation, we compute:
    \[
        \sum^{n-2}_{i=2} (n - i)s_i = -(n^2 - 3n) + 2n(k_2 - k_{n-1}) - \sum^{n-2}_{i=2} i\left(k_i - k_{i+1} + k_{n-i} - k_{n-i+1}\right).
    \]
    Simplifying the right-hand summation through index reparameterization reveals telescoping cancellations:
    \[
        \sum^{n-2}_{i=2} i\left(k_i - k_{i+1} + k_{n-i} - k_{n-i+1}\right) = (n - 1)(k_2 - k_{n-1}) + \sum^{n-3}_{i=1} (k_{i+1} - k_{n-i}),
    \]
    which reduces~\eqref{e:son} to \( k_1 + k + l = 0 \). 
    
    Since \( k_1, k, l \) are non-negative integers,  the previous claim forces \( k_1 = k = l = 0 \), whic is  a contradiction. Thus, \( S \) cannot have a node. 
\end{proof}
We next constrain the possible inflection configurations:

\begin{corollary}[Inflection set] \label{c:inflection}
Let \( g \in \mathrm{PGL}(n+1,\mathbb{C}) \) with \( n \geq 2 \) be an element of infinite order, and let \( S \subset \mathbb{CP}^n \) be a non-degenerate, irreducible rational algebraic curve invariant under \( g \). Suppose \( p \in S \) is an inflection point which is also fixed point by \( g \), and let \( \mathcal{L} \subset \mathbb{CP}^n \) be a hyperplane not containing \( p \). If the projection \( S_{p,\mathcal{L}} \) is projectively equivalent to a monomial set, then  the number of inflection in \( S \) is  2.
\end{corollary}

\begin{proof}
    Suppose, for contradiction, that \( p \) in the unique   inflection point for \( S \). 

As in  the proof of  Corollary~\ref{c:node}, by  Lemma~\ref{l:para} we can assume that   \( p = [e_{n+1}] \), $\mathcal{L}=\langle \langle [e_1],\ldots,[e_n]\rangle\rangle$,  and there exist: homogeneous polynomials \( P(z,w), Q(z,w) \in \mathbb{C}[z,w] \), a sequence of natural numbers \( k_1 > k_2 > \ldots > k_{n-1} \), and a desingularization \( \psi: \mathbb{CP}^1 \to S \) satisfying \( \psi^{-1}(p) = \{[0,1], [1,0]\} \) and
\begin{scriptsize}
    \[
        \psi[z,w] = \left[ z^{k_1}P(z,w):\, z^{k_2}w^{k_1 - k_2}P(z,w):\, \ldots:\, z^{k_{n-1}}w^{k_1 - k_{n-1}}P(z,w):\, w^{k_1}P(z,w):\, Q(z,w) \right].
    \]
    \end{scriptsize}
Since      $S$ does not contain nodes (see Corollary \ref{c:node})  and    \( \psi^{-1}(p) \subseteq \{[0:1], [1:0]\} \) we have either \( P(z, w)= az^k \) or \(  P(z, w)= bw^l\). 
To complete the proof, we split into two exhaustive cases.

\medskip

\noindent\textbf{Case 1.}  Suppose 
\[
P(z,w)=a\,z^k.
\] 
Then 
\[
\psi[z:w]
=\bigl[a\,z^{k_1+k}:a\,z^{k_2+k}w^{\,k_1-k_2}:\cdots:
a\,z^{k_{n-1}+k}w^{\,k_1-k_{n-1}}:a\,w^{k_1}z^k:Q(z,w)\bigr].
\]
On the affine chart \(\{z\neq0\}\) we set \(w=1\); the lift becomes
\[
\widetilde\psi(z)
=\bigl(a\,z^{k+k_1},\,a\,z^{k+k_2},\,\dots,\,a\,z^{k+k_{n-1}},\,a\,z^k,\;r(z)\bigr),
\quad r(z)=Q(z,1).
\]
Since \(S\) has exactly one inflection point, the orders of contact \(s_j\) at that point satisfy
\[
\begin{aligned}
s_0 &= (k)-1,\\
s_1 &= (k_{n-1})-1,\\
s_j &=\bigl(k_{n-j}-k_{n-j+1}\bigr)-1
\quad(2\le j\le n-1).
\end{aligned}
\]
Substituting these into Plücker’s relation (Lemma \ref{l:1.2}) gives 
\[
r_{n-1}=n.
\]
Hence the dual curve \(S^\vee\subset\mathbb{CP}^n\) is the rational normal curve (see \cite[p.\,263]{GH}).  By duality \(S^{\vee\vee}\) is also rational normal, and then the Reflexivity Theorem (Theorem 15.24 in \cite{H}) forces $S=S^{\vee\vee},$ contradicting the fact that \(S\) has an inflection point.

\medskip

\noindent\textbf{Case 2.}  Suppose 
\[
P(z,w)=b\,w^l.
\]
An analogous argument—interchanging the roles of  $z$ and $w$—applies verbatim. We therefore omit the details.

\medskip

This completes the proof.

\end{proof}

\begin{corollary} \label{c:inflection2}
    Let $S \subset \mathbb{CP}^n$ be an irreducible algebraic curve with a desingularization $\psi: \P^1 \rightarrow \mathbb{CP}^n$ given explicitly by the morphism:
    $$
    \left[z^{k_1+k}: z^{k_2+k}w^{k_1-k_2}P(z,w): \ldots:\ z^{k_{n-1}+k}w^{k_1-k_{n-1}}:\ z^k w^{k_1},\ Q(z,w)\right]
    $$
    where $Q(z,w) \in \C[z,w]$ is a homogeneous polynomial of degree $k_1 + k$ satisfying $Q(0,1) \neq 0$ and $Q(1,0) = 0$. Assume $[1:0]$ and $[0:1]$ are the only W-points of $\psi$. Then there exist $b \in \C^*$ and $v \in \C^{n-1}$ such that $h \circ \psi = \xi_{\mathbf{k}}$, where $\mathbf{k} = (k_1 + k, \ldots, k_{n-1} + k, k)$ and 
    $$
    h = 
    \begin{bmatrix}
        1 & 0 & 0 \\
        0 & I_{n-1} & 0 \\
        0 & v & b
    \end{bmatrix}.
    $$
\end{corollary}

\begin{proof}
    The morphism $\psi$ lifts to local expressions in the affine charts:
    \[
    \begin{array}{ll}
        \text{Chart } \{(z,1)\}: & \left( z^{k+k_1}, z^{k+k_2}, \ldots, z^{k+k_{n-1}}, z^k, r(z) \right), \\[0.5em]
        \text{Chart } \{(1,w)\}: & \left( 1, w^{k_1 - k_2}, \ldots, w^{k_1 - k_{n-1}}, w^{k_1}, s(w) \right),
    \end{array}
    \]
    where $r(z) = Q(z,1)$ and $s(w) = Q(1,w)$. The condition $Q(1,0) = 0$ implies $s(0) = 0$, so $s(w)$ expands as a sum $s(w) = \sum_{j=1}^m a_{l_j} w^{l_j}$ with strictly increasing exponents $1 \leq l_1 < \cdots < l_m = k + k_1$. 

    Consider the transformation $h \circ \psi$ with 
    $$
    h = 
    \begin{bmatrix}
        1 & 0 & 0 \\
        0 & I_{n-1} & 0 \\
        0 & v & 1
    \end{bmatrix}.
    $$
    By suitably choosing $v \in \C^{n-1}$, we may eliminate all terms in $s(w)$ whose exponents $l_j$ lie in the set $\{k_1 - k_2, \ldots, k_1 - k_{n-1}, k_1\}$. Let $\{\alpha_i\}_{i=1}^n$ be the ordered sequence ($\alpha_1 < \cdots < \alpha_n$) formed by the exponents $\{l_1, k_1 - k_2, \ldots, k_1 - k_{n-1}, k_1\}$, and define $\beta_j = \alpha_j + k + k_{j-1}$ for $1 \leq j \leq n$ (with $k_0 = 0$). Direct computation yields the relations:
    $$
    s_i = \beta_{i+1} - \beta_i - 2 \quad \text{for } 1 \leq i \leq n - 1.
    $$
    
    Substituting $j = n-2$ and $j = n-1$ into Equation \ref{e:plud}, we derive:
    \[
    \begin{aligned}
        r_{n-1} &= n(k + k_1 - n + 1) - \sum_{i=0}^{n-2} (n - 1 - i)s_i, \\
        r_{n-2} &= (n - 1)(k + k_1 - n + 2) - \sum_{i=0}^{n-3} (n - 2 - i)s_i.
    \end{aligned}
    \]
    The identity $r_{n-2} - 2r_{n-1} = -2 - s_{n-1}$ simplifies to:
    \begin{equation}\label{e:pluc3}
        -(k_1 + k) - n(n + 1) + n \max\{k_1, l_1\} - \sum_{i=1}^{n-1} \sum_{j=i}^{n-1} s_j = 0.
    \end{equation}
Expanding the double sum, we compute:
$$
    \sum_{i=1}^{n-1} \sum_{j=i}^{n-1} s_j = (n - 1)(k + k_1 + \max\{k_1, l_1\}) - n(n - 1) - l_1 - n(k + k_1).
    $$
    Substituting this into Equation \ref{e:pluc3} forces $l_1 = k + k_1$, completing the proof.
\end{proof}

Synthesizing these constraints yields our main classification:

\begin{lemma}[Monomial Equivalence] \label{l:cc}
    Let \( g \in \mathrm{PGL}(n+1, \mathbb{C}) \) be an element of infinite order, and let \( S \subset \mathbb{CP}^n \) be a non-degenerate irreducible rational complex algebraic curve invariant under \( g \). Then, \( S \) is projectively equivalent to a monomial set.
\end{lemma}
\begin{proof}
We proceed by induction on \( n \). 
For the base case \(n=2 \), first suppose that \( S \) is the Veronese curve. In this case, the result follows immediately. If \( S \) is not the Veronese curve, let $p\in S$ be a singular point and $\mathcal{L}$ a line not containing $p$, then in this case by Bézout's Lemma we deduce $S_{p,\mathcal{L}}$ is a line, by Corollaries \ref{c:node} and \ref{c:inflection} we deduce that $S$ has exactly two  inflection points and no nodes.

Now, Applying Lemma \ref{l:para}, we may assume \( p = [e_3] \), \( \mathcal{L} = \langle [e_1], [e_2] \rangle \), and that the inflection points of \( S \) are \( \{[e_1], [e_3]\} \). This lemma also ensures a desingularization \( \psi \colon \mathbb{CP}^1 \to S \) satisfying \( \psi^{-1}(\{[e_1], [e_3]\}) = \{[1:0], [0:1]\} \), with \( \psi \) explicitly given by:
\[
\psi([z,w]) = \left[z^kP(z,w):\, w^kP(z,w):\, Q(z,w)\right],
\]
 where  \( k \in \mathbb{N} \),  \( P(z,w), Q(z,w) \in \mathbb{C}[z,w] \), are  homogeneous polynomials. As in Corollary \ref{c:inflection}  the preimage  condition \( \psi^{-1}(\{[e_1], [e_3]\}) = \{[1:0], [0:1]\} \) forces \( P(z,w) \) to be a monomial, either \( az^{k_1} \) or \( aw^{k_1} \), for some \( a \in \mathbb{C} \). To conclude this part of the proof,  observe that applying Corollary \ref{c:inflection2} completes the argument for \(n=2 \).

\medskip

The inductive step follows analogously, and we omit its proof for brevity.
\end{proof}

				{\bf Proof of Theorem \ref{t:main1}} This follows easily from theorem Lemma \ref{l:cc}. $\square$ \\

\section{Projective Automorphism Groups of Monomial Curves  } \label{s:examples}

In this section, we characterize the projective automorphism groups of algebraic curves parametrized by monomial curves. Our approach proceeds through a sequence of technical lemmas, culminating in our main theorem of classification. We begin by establishing key properties of curve projections.

\begin{lemma}

Let \( \mathbf{k} = (k_1, \dots, k_n) \in \mathbb{N}^n \) satisfying \(k_1 > \cdots > k_{n} > 0\). If \( S^{(\mathbf{k})} \subset \mathbb{CP}^n \) denotes the image of the monomial curve \( \xi_{\mathbf{k}} \), and  we define the hyperplanes
\[
\mathcal{L}_1 = \langle \langle [e_2], \dots, [e_{n+1}] \rangle\rangle \quad \text{and} \quad \mathcal{L}_2 = \langle\langle [e_1], \dots, [e_n] \rangle\rangle.
\]
Then the curve projections \( S^{(\mathbf{k})}_{[e_1],\mathcal{L}_1} \) and \( S^{(\mathbf{k})}_{[e_{n+1}],\mathcal{L}_2} \) can each be parametrized through monomial curves.
\end{lemma}

\begin{proof}
Let \([z, w] \in \mathbb{CP}^1 - \{[1: 0], [0: 1]\}\). The projections are explicitly computed as:
\begin{small}
\[
\begin{aligned}
\pi_{[e_1], \mathcal{L}_1}\big(\xi_{\mathbf{k}}([z: w])\big) &= \left[ z^{k_2}: \, z^{k_3}w^{k_2 - k_3}: \, \dots: \, z^{k_n}w^{k_2 - k_n}: \, w^{k_2} \right] = \xi_{(k_2, \dots, k_n)}([z: w]), \\
\pi_{[e_{n+1}], \mathcal{L}_2}\big(\xi_{\mathbf{k}}([z: w])\big) &= \left[ z^{k_1 - k_n}: \, z^{k_2 - k_n}w^{k_1 - k_2}: \, \dots:\, w^{k_1 - k_n} \right] = \xi_{(k_1 - k_n, \dots, k_{n-1} - k_n)}([z: w]).
\end{aligned}
\]
\end{small}
These equalities demonstrate that both projections correspond to monomial curves in their respective hyperplanes, thereby concluding the proof.\end{proof}

\begin{definition}
Let \(\mathbf{k} = (k_1, \dots, k_n) \in \mathbb{N}^n\) satisfy \(k_1 > \cdots > k_{n} > 0\) and set  \(S^{(\mathbf{k})} \subset \mathbb{CP}^n\) to be the image of the monomial curve \(\xi_{\mathbf{k}}\). We define the stabilizer group of \(S^{(\mathbf{k})}\) as:
\[
\mathcal{V}_{\mathbf{k}} = \left\{ g \in \mathrm{PSL}(n+1, \mathbb{C}) \mid g \cdot S^{(\mathbf{k})} = S^{(\mathbf{k})} \right\},
\]
where \(\mathrm{PSL}(n+1, \mathbb{C})\) acts on \(\mathbb{CP}^n\) via its standard projective linear action.
\end{definition}

\begin{lemma}
Let \( \mathbf{k} = (k_1, \dots, k_n) \in \mathbb{N}^n \) with \( k_1 > n \) and \( \gcd(k_1, \dots, k_n) = 1 \). Then  the group \( \mathrm{Isot}(\mathcal{V}_{\mathbf{k}}, [e_1]) \)—that is, the subgroup of \( \mathcal{V}_{\mathbf{k}} \) fixing \( [e_1] \)—agrees with the group \( \mathrm{Isot}(\mathcal{V}_{\mathbf{k}}, [e_{n+1}]) \).
\end{lemma}
\begin{proof}
The proof is straightforward and is therefore omitted.
\end{proof}

\begin{lemma}\label{l:cc2}  
Let \( \mathbf{k} = (k_1, \dots, k_n) \in \mathbb{N}^n \) satisfying \( k_1 > n \) and \(\gcd(k_1, \dots, k_n) = 1\). Then every element \( g \in \mathrm{Isot}(\mathcal{V}_{\mathbf{k}}, [e_1]) \) (the subgroup of \( \mathcal{V}_{\mathbf{k}} \) fixing \( [e_1] \)) admits a diagonal lift in \( \mathrm{GL}(n+1, \mathbb{C}) \).
\end{lemma}

\begin{proof}  
We proceed by induction on \( n \).  

Let us show the base case (\( n = 2 \)).  
A direct computation shows that the tangent line to the rational curve \( \xi_{\mathbf{k}}(\mathbb{CP}^1) \) at \( [e_3] \) is uniquely \( \langle\!\langle [e_2], [e_3] \rangle\!\rangle \), and at \( [e_1] \) it is \( \langle\!\langle [e_1], [e_2] \rangle\!\rangle \). Since \( \mathrm{Isot}(\mathcal{V}_{\mathbf{k}}, [e_1]) \) acts on \( \mathbb{CP}^2 \) by biholomorphisms, it must preserve these tangent lines, and in consequence $[e_3]$. Consequently, the subspaces \( \langle\!\langle e_2, e_3 \rangle\!\rangle \) and \( \langle\!\langle e_1, e_3 \rangle\!\rangle \) are \( \mathrm{Isot}(\mathcal{V}_{\mathbf{k}}, [e_1]) \)-invariant. This invariance forces \( [e_2] \) to be fixed by the entire isotropy subgroup.  

To show the inductive step,  
assume the result holds for \( n-1 \). Let us define \( \mathcal{L}_1 = \langle\!\langle [e_2], \dots, [e_{n+1}] \rangle\!\rangle \) and \( \mathcal{L}_2 = \langle\!\langle [e_1], \dots, [e_n] \rangle\!\rangle \). We analyze three cases:  

\textbf{Case 1:} \( S_{[e_1], \mathcal{L}_1}^{(\mathbf{k})} \) is the normal rational curve, while \( S_{[e_{n+2}], \mathcal{L}_2}^{(\mathbf{k})} \) is not projectively equivalent to the  normal rational curve.  
By the inductive hypothesis, any \( g \in \mathrm{Isot}(\mathcal{V}_{\mathbf{k}}, [e_1]) \) has a block form:  
\[
g = 
\begin{pmatrix}
a_1 & 0 & 0 \\
0 & D & 0 \\
0 & u & a_2
\end{pmatrix}, \quad \text{where } D \in \mathrm{GL}(n-1, \mathbb{C}) \text{ is diagonal, } a_1, a_2 \in \mathbb{C}^*, \text{ and } u \in \mathbb{C}^{n-1}.
\]  
Since \( h \) fixes \( [e_{n+1}] \) and the embedding \( \iota_{n-1} \) is \( \mathrm{PGL}(2, \mathbb{C}) \)-equivariant, there exist \( x, y \in \mathbb{C} \) such that  
\[
\iota_{n-1}\left( \begin{pmatrix} x & 0 \\ y & x^{-1} \end{pmatrix} \right)  = \begin{pmatrix} D & 0 \\ u & a_2 \end{pmatrix}.
\]  
Examining the \((2,1)\)-entry of the matrix equation (from Formula \eqref{e:alex}), we find \( y = 0 \). This implies \( u = 0 \), rendering \( g \) diagonal.  

\textbf{Case 2:} Both \(  S_{[e_1], \mathcal{L}_1}^{(\mathbf{k})}  \) and \( S_{[e_{n+2}], \mathcal{L}_2}^{(\mathbf{k})}  \) are normal rational curves.  
Here, \( g \in \mathrm{Isot}(\mathcal{V}_{\mathbf{k}}, [e_1]) \) has the form:  
\[
g = 
\begin{pmatrix}
a_1 & v & 0 \\
0 & D & 0 \\
0 & u & a_2
\end{pmatrix}, \quad \text{with } v \in \mathbb{C}^{n-1}, \, D \in \mathrm{GL}(n-1, \mathbb{C}).
\]  
By equivariance of \( \iota_{n-1} \), there exist \( x, y, \tilde{x}, \tilde{y} \in \mathbb{C} \) such that:  
\[
\iota_{n-1}\left( \begin{pmatrix} x & 0 \\ y & x^{-1} \end{pmatrix} \right) = \begin{pmatrix} D & 0 \\ u & a_2 \end{pmatrix}, \quad 
\iota_{n-1}\left( \begin{pmatrix} \tilde{x} & \tilde{y} \\ 0 & \tilde{x}^{-1} \end{pmatrix} \right) = \begin{pmatrix} a_1 & v \\ 0 & D \end{pmatrix}.
\]  
From Formula \eqref{e:alex}, the \((2,1)\)-entry of the first equation forces \( y = 0 \), hence \( u = 0 \). Similarly, the \((2,n)\)-entry of the second equation implies \( \tilde{y} = 0 \), giving \( v = 0 \). Thus, \( D \) must be diagonal, and \( g \) is diagonal.  

\textbf{Case 3:} Neither \(  S_{[e_1], \mathcal{L}_1}^{(\mathbf{k})}  \) nor \(  S_{[e_{n+2}], \mathcal{L}_2}^{(\mathbf{k})}  \) is projectively equivalent to the normal rational curve.  
The argument combines methods from Cases 1 and 2: invariance of subspaces under the isotropy action and constraints from the equivariance of \( \iota_{n-1} \) similarly force \( u = 0 \), \( v = 0 \), and diagonality of \( D \). We omit repetitive details.  

In all cases, \( g \) admits a diagonal lift, completing the induction.  
\end{proof}

 With these projection properties established, we now formalize the symmetry notion for elements in $\mathbb{N}^n$ central to our investigation:

\begin{definition}\label{def:symmetric_tuple}
A tuple \(\mathbf{k} = (k_1, \ldots, k_n) \in \mathbb{Z}^n\) is called \emph{symmetric} if it satisfies:
\begin{itemize}
    \item The entries are strictly decreasing positive integers with no common divisor, that is:
    \[
    k_1 > k_2 > \cdots > k_n \geq 1 \quad \text{and} \quad \gcd(k_1, \ldots, k_n) = 1.
    \]
    
    \item  The differences between consecutive entries are symmetric. Specifically, for all \(0 \leq j \leq n-2\),
    \[
    k_{n-j} - k_{n-j+1} = k_{j+1} - k_{j+2},
    \]
    where we define \(k_{n+1} = 0\) to extend the tuple. 
\end{itemize}
\end{definition}

\begin{theorem}\label{c:aut}
Let \(\mathbf{k} = (k_1, \dots, k_n) \in \mathbb{N}^n\) satisfying     \(k_1 > k_2 > \cdots > k_n \geq 1\),  \(k_1 > n\) and \(\gcd(k_1, \dots, k_n) = 1\). Then:
\begin{enumerate}
    \item If \(\mathbf{k}\) is symmetric, the automorphism group \(\mathcal{V}_{\mathbf{k}}\) is generated by the holomorphic automorphism group \(\mathcal{H}_{\mathbf{k}}\) and the equivalence class of the anti-diagonal matrix \(J_n\) ({\it i.e.} an \((n+1)\times(n+1)\) matrix with 1s on the anti-diagonal), in other words:
    \[
    \mathcal{V}_{\mathbf{k}} = \left\langle \mathcal{H}_{\mathbf{k}} ,[J_n] \right\rangle.
    \]
    
    \item If \(\mathbf{k}\) is not symmetric, then \(\mathcal{V}_{\mathbf{k}}=\mathcal{H}_{\mathbf{k}} \).
\end{enumerate}
\end{theorem}

\begin{proof}
We first analyze the isotropy group \(\operatorname{Isot}([e_1], \mathcal{V}_{\mathbf{k}})\).

\textbf{Claim :} \(\operatorname{Isot}([e_1],\mathcal{V}_{\mathbf{k}})=\mathcal{H}_{\mathbf{k}}\).  
   By Lemma 3.4, any \(g \in \operatorname{Isot}([e_1], \operatorname{V}_{\mathbf{k}})\) is represented by a diagonal matrix \([\operatorname{diag}(\rho_1, \ldots, \rho_{n+1})]\). The invariance of the monomial curve \(\xi_{\mathbf{k}}(\mathbb{CP}^1)\) under \(g\) implies that for each \([z, w] \in \mathbb{CP}^1\), there exists \([a, b] \in \mathbb{CP}^1\) such that:
    \[
    \rho_j u^{k_j} = v^{k_j} \rho_{n+1} \quad \text{for } 1 \leq j \leq n,
    \]
    where \(wu = z\) and \(vb = a\). Algebraic manipulation  yields:
    \[
    \rho_{n+1}^{k_1 - k_j} \rho_1^{k_j} = \rho_j^{k_1} \quad \text{for all } j.
    \]
    Let \(\rho_1 = \alpha^{k_1}\) and \(\rho_{n+1} = \beta^{k_1}\) for \(\alpha, \beta \in \mathbb{C}^*\). Substituting these into the equations gives:
    \[
    g = \left[\operatorname{diag}\left(\alpha^{k_1}, \mu_2 \alpha^{k_2}\beta^{k_1 - k_2}, \ldots, \mu_n \alpha^{k_n}\beta^{k_1 - k_n}, \beta^{k_1}\right)\right],
    \]
    where \(\mu_j\) are \(k_1\)-th roots of unity. Setting \(z = w = 1\) simplifies the relations, forcing \(\mu_j = 1\) (by coprimality). Thus, \(g\) is determined by \(\alpha\) and \(\beta\), proving the claim.

    Finally, the inflection set of the monomial curve \(\xi_{\mathbf{k}}\) is \(\{[e_1], [e_{n+1}]\}\). Since \(\mathcal{V}_{\mathbf{k}}\) acts by biholomorphisms, any \(g \in \mathcal{V}_{\mathbf{k}}\) must either fix both \([e_1]\) and \([e_{n+1}]\) or swap them. We must consider  two cases:

\begin{itemize}
    \item \textbf{Case 1:} If \(g\) swaps \([e_1]\) and \([e_{n+1}]\), then \(g(e_1) = e_{n+1}\). Coprimality of \(\mathbf{k}\) enforces symmetry in the differences \(k_{n-j} - k_{n-j+1} = k_{j+1} - k_{j+2}\) (via equating weights for the transformed curve), which holds only if \(\mathbf{k}\) is symmetric.
    
    \item \textbf{Case 2:} If \(\mathbf{k}\) is symmetric, the matrix \(J_n\) explicitly swaps the coordinates \([z, w] \leftrightarrow [w, z]\), satisfying:
    \[
    [J_n] \xi_{\mathbf{k}}([a, b]) = \xi_{\mathbf{k}}([b, a]).
    \]
    Hence, any \(g \in \mathcal{V}_{\mathbf{k}}\) either preserves \([e_1]\) (and lies in \(\operatorname{Isot}([e_1], \operatorname{H}_{\mathbf{k}})\)) or is a composition of an element of \(\operatorname{Isot}([e_1], \operatorname{H}_{\mathbf{k}})\) with \([J_n]\).
\end{itemize}

Thus,  
if \(\mathbf{k}\) is symmetric, \(\mathcal{V}_{\mathbf{k}}=\langle \operatorname{H}_{\mathbf{k}},[J_n]\rangle \). Otherwise, no swapping occurs, and \(\operatorname{H}_{\mathbf{k}} = \mathcal{V}_{\mathbf{k}}\).  
\end{proof}

		{\bf Proof of Corollary \ref{t:main2}} This is an easy consequence of Theorems \ref{t:main1} and  \ref{c:aut}.$\square$\\

\begin{center}
    \scshape \bfseries Funding
\end{center}
The research of A. Cano was partially supported by SECIHTI-SNII 104023 and PAPPIT-UNAM IN112424, and the research of L. Loeza was partially supported by  SECIHTI-SNII 43823 and PAPPIT-UNAM IN112424.

\begin{center}
    \scshape \bfseries Declarations
\end{center}
All authors have contributed equally to the paper and declare no conflicts of interest.

\begin{center}
    \scshape \bfseries Acknowledgments
\end{center}
The authors thank the people of UCIM-UNAM  and IIT-UACJ  for their hospitality during the preparation of this paper. We also thank A. D. Rios, J. A. Seade, and M. A. Ucan for their valuable discussions.

  \bibliographystyle{amsplain}

\begin{thebibliography}{10}
		


\bibitem{BCNS}
Barrera, W.;  Cano, A.;  Navarrete, J. P.;  Seade, J.: 
\emph{ On Sullivan’s dictionary in complex dimension	two}, Preprint 2025.


\bibitem{BCS}
Briend, J. Y.;     Cantat, S.;   Shishikura, M.: 
\emph{  Linearity of the exceptional set for maps of}  $P_ k (C)$,
M. Math  Ann. (2004) 330, Issue  1,   39-43.
			

\bibitem{BK}
Brieskorn, E.;  Knörrer, H.:
\emph{ Plane algebraic curves}, 
Birkhäuser Basel, 1986.

            
            
\bibitem{CLU} 	
Cano, A.;  Loeza, L.; Ucan-Puc, A.:
\emph{ Projective cyclic groups in higher dimensions}, 
Linear Algebra and its Applications, 	531 (2017),  169-209. 
			
\bibitem{CNS}
Cano, A.;  Navarrete, J. P.;  Seade, J.: 
\emph{ Complex Kleinian Groups},  Progress in Mathematics, volume 303, Birkhäuser Basel, 2012.

            
\bibitem{CN}
Cerveau, D.;   Neto, A. L.: 
\emph{ Holomorphic foliations in} $\mathbb{CP}(2)$  \emph{having an invariant algebraic curve},
Annales de l'Institut Fourier, Tome 41 (1991) no. 4,  883-903.

              

\bibitem{C}
Chirka, E. M.: 
\emph{ Complex analytic sets}, 
MASS, volume 46, Kluwer Academic Publishers, 1989.	

            
\bibitem{GH}
Griffiths, P.;   Harris, J.: 
\emph{Principles of algebraic geometry},  John Wiley \& Sons, 1994.
      


\bibitem{GR}
Guivarc'h, Y.;    Raugi, A.:
\emph{ Actions of large semigroups and random walks on isometric extensions of boundaries},
Ann. Sci. Ecole Norm  Sup. (4) 40 (2007),  209-249.
		


\bibitem{H}
Harris, J.: 
\emph{Algebraic geometry, a first course},  
Springer-Verlag, New York, 1992. 


\bibitem{Ho}
Höring, A.:
\emph{ Totally invariant divisors of endomorphisms of projective spaces},
Manuscripta math. 153, 173–182 (2017). 


\bibitem{HN}
Hwang, J. H.;  Nakayama,   N.:
\emph{ On endomorphisms of Fano manifolds of Picard number one},
Pure Appl. Math. Q., 7(4, Special Issue: In memory of Eckart Viehweg), 1407–1426, 2011.


               
\bibitem{J}
Jouanolou, J. P.:
\emph{ Equations of Pfaff algébriques}, 
Lect. Notes in Math 708, Springer Verlag, Berlin, 1979.


\bibitem{M}
Miranda, R.:  
\emph{  Algebraic curves and Riemann surfaces},  
Graduate Studies in Mathematics Volume 5, AMS, 1995.



\bibitem{MM}
Matsumura, H.; Monsky, P.:
\emph{On the automorphisms of hypersurfaces}.
J. Math. Kyoto Univ. 3(3): 347-361 (1963). DOI: 10.1215/kjm/1250524785
            


            
\bibitem{P}
Popov, V. L.: 
\emph{ Algebraic curves with an infinite automorphism group},
Mathematical Notes of the Academy of Sciences of the USSR 23, 102–108, 1978.



            
\bibitem{SV}
Seade, J.;  Verjovsky, A.: 
\emph{ Higher dimensional complex Kleinian groups}, Math. Ann., 322, 279-300.
			
			
\bibitem{V}
Videnskii, I. V.: 
\emph{Zeros of the derivative of a rational function and coinvariant subspaces for the shift operator on the Bergman space},  
Journal of Mathematical Sciences 120, 1657–1661 (2004).
			

\bibitem{W}
Wall, C. T. C.: 
\emph{ Plücker formulae for curves in high dimensions},
Rend. Lincei Mat. Appl. 20 (2009), 159–177.



\bibitem{W1}
Wall, C. T. C.: 
\emph{ Projection genericity of space curves}, Journal of topology,  Volume 1, Issue2, 2008
362-390.


\end{thebibliography}

	\end{document}